\documentclass{amsart}

\usepackage{hyperref}
\usepackage{stmaryrd}

\newtheorem{theorem}{Theorem}[section]
\newtheorem{lemma}[theorem]{Lemma}
\newtheorem{corollary}[theorem]{Corollary}
\newtheorem*{claim}{Claim}
\newtheorem*{pp}{Proposition}

\theoremstyle{definition}
\newtheorem{definition}[theorem]{Definition}
\newtheorem{example}[theorem]{Example}
\newtheorem{p}[theorem]{Problem}
\newtheorem{prop}[theorem]{Proposition}

\theoremstyle{remark}
\newtheorem{remark}[theorem]{Remark}

\numberwithin{equation}{section}

\begin{document}

\title{Area-minimizing Hypersurfaces in Singular Ambient Manifolds}


\author{Yihan Wang}
\address{Peking University, Beijing, China}
\curraddr{}
\email{2001110016@pku.edu.cn}
\thanks{}

\def\ZZ{{\mathbb Z}}
\def\QQ{{\mathbb Q}}
\def\RR{{\mathbb R}}
\def\CC{{\mathbb C}}
\subjclass[2020]{28A75, 53A10, 53C23}

\keywords{Area-minimizing hypersurface, regularity,
tangent cone, nonnegative scalar curvature, singular ambient space.}

\date{}

\dedicatory{}

\begin{abstract}
We study area-minimizing hypersurfaces in singular ambient manifolds whose tangent cones have nonnegative scalar curvature on their regular parts. We prove that the singular set of the hypersurface has codimension at least 3 in our settings. We also give an example in which $n$-dimensional minimizing hypersurface in an ambient space as given above may have singularities of dimension $n-3$. This shows that our regularity result is sharp.
\end{abstract}

\maketitle

\section{Introduction}

Regularity of area-minimizing surfaces has been widely studied since the last century, particularly, regularity of minimizing hypersurfaces in a smooth ambient space was well understood, and it has been proved (see \cite{Fleming1962},\cite{Simons1968},\cite{Bombieri1969},\cite{Federer1970}) that singularities of minimizing hypersurfaces are of codimension at least 7 when the ambient space is a smooth $(n+1)$-dimensional manifolds, so when $n\geq 7$, a minimizing hypersurface may have singularities. Recently, regularity of generic minimizing hypersurfaces have been studied and important results have been established when $n$ is not more than $9$ thanks to works of \cite{Chodosh2023new} and \cite{Chodosh2023} (also see \cite{Smale1993}). We also like to mention that the theory of minimizing hypersurfaces have had many important applications, most notably, it has been used in proving positive mass theorems (see \cite{Yau1979}, \cite{Schoen1979} and \cite{Yau2017}).

In this paper, we will generalize classical regularity theory for area-minimizing hypersurfaces to the case of singular ambient spaces. We will see that because of presence of singularity in ambient spaces, regularity of minimizing hypersurfaces differs from the smooth case in a substantial way.

Let us start with a well-known example. Consider a 2-dimensional cone $C$ with an isolated singular point $O$ (the vertex) whose cross section is a closed Jordan curve $\gamma$. Given any two points $p,q$ on $\gamma$, if the length of $\gamma$ is strictly less than $2\pi$, the minimizing geodesic joining $p$ and $q$ is smooth, that is, it does not pass through the singular point $O$. How about area-minimizing hypersurfaces in spaces with conic singularity? Note that the 2-dimensional cone $C$ is developable, so it is flat, and it is the curvature of the cross section of $C$ that counts in this minimizing problem. Therefore, a natural generalization is to consider whether an area-minimizing hypersurface in an $(n+1)$-dimensional cone contains the singular vertex of the cone, assuming some additional conditions on the ambient space. As a first step, we will discuss whether there exists area-minimizing hypercones in 3-dimensional cones. And we will prove a new nonexistence theorem of such 2-dimensional minimizing hypercones assuming the ambient space has nonnegative scalar curvature (see Theorem \ref{thm3.2}).

Recall that the non-existence of stable minimal hypercones in $\RR^{n+1}$ $(n\le 6)$ plays an important role in classical regularity theory of area-minimizing hypersurfaces (see \cite{Simons1968}). Here, analogously, we can obtain a regularity result for area-minimizing hypersurfaces in singular ambient spaces, using Federer’s dimension reduction argument. However, some other difficulties arise in this process when generalizing classical results from smooth ambient spaces to singular ambient spaces. So modifications are necessary here, for instance, the compactness (Lemma \ref{cpt}), the monotonicity formula (Lemma \ref{mono}), the density estimates (Lemma \ref{lowerbound}), and so on. These results will be proved using the language of integral currents in geometric measure theory (see \cite{Simon1983}) in the following sections.  

These new challenges lead us to introduce the following class $\mathcal{F}=\cup_n\mathcal{F}_n$ of singular spaces which are natural extensions of 2-dimensional flat cones and 3-dimensional scalar nonnegative cones. We will introduce a family ${\mathcal F}_n$ of admissible spaces which may have singularities. First ${\mathcal F}_n$ contains all the smooth Riemannan manifolds of dimension $n+1$ as well as all the metric cones $C(\Sigma)$ of dimension $n+1$ such that its vertex $O$ is the only singularity and its scalar curvature on the regular part $C(\Sigma)\backslash \{O\}$ is non-negative. (If $n=1$, the cone angle is assumed to be no larger than $2\pi$, additionally.) By a metric cone, we mean a complete metric space $C$ which is isometrically embedded into some Euclidean space and satisfies: $\eta_{0,\lambda,\#}C=C$ for all $\lambda>0$, where $\eta_{p,\lambda,\#}$ is the scaling (see Section 2.3). Of course, if $C$ is a $k$-dimensional cone ($k< n+1$) with isolated singularity and non-negative scalar curvature outside the vertex, then $\RR^{n+1-k}\times C$ belongs to ${\mathcal F}_n$. Then we can define inductively other admissible spaces. Suppose that we have defined ${\mathcal F}_{n-1}$, then we can define ${\mathcal F}_n$ as follows: For any $M'\in {\mathcal F}_{n-1}$, the product $\RR\times M'\in {\mathcal F}_n$ is admissible. We say that $M\in\mathcal{F}_n$ if, roughly speaking, $M$ has scalar nonnegative tangent cones with good structures, and there are estimates on the densities and singular sets of $M$. The precise definition of $\mathcal{F}_n$ will be given in the next section (see Definition \ref{admissible}).

It is not hard to see that admissible spaces in ${\mathcal F}$ include a large number of RCD($0,N$) metric measure spaces, projective varieties with only the klt singularities as well as Gromov Hausdorff limit spaces of Riemannian manifolds with uniformly bounded Ricci curvature. If the elements of $\mathcal{F}_n$ are all area-minimizing hypersurfaces of $(n+2)$-manifolds, then almost all the conditions are satisfied, e.g., if $\Sigma$ is a minimizing hypersurface in $\RR^8$, then classical result implies that $\Sigma $ can only have isolated singularities and satisfies the hypotheses on tangent cones in the definition of $\mathcal{F}_6$.

The following regularity theorem is for singular ambient spaces and will be our main result:
\begin{theorem}\label{main1}
For $M^{n+1}\in\mathcal{F}_n$ (see Definition \ref{admissible}), let $\Sigma^n$ be an area-minimizing hypersurface in $M^{n+1}$. If (6a) in Definition \ref{admissible} holds, then the singular points of $\Sigma^n$ form a set of Hausdorff dimension at most $n-3$. In particular, $\Sigma^2$ is always smooth in $M^3$.
\end{theorem}
\begin{remark}
    Since the condition (6a) is equivalent to saying that $M$ itself has codimension 3 singularities, the theorem demonstrates that the minimizing hypersurface inherits this codimension 3 property from the ambient space, which, consequently, makes it possible to start an induction process when considering a sequence of dimension decreasing minimizing hypersurfaces.
    
    If (6b) holds instead of (6a) in Definition \ref{admissible}, then the Hausdorff dimension of $\text{Sing}(\Sigma^n)$ is no larger than $n-2$, respectively. Moreover, if we drop the assumption (7) in the definition of $M^{n+1}\in \mathcal{F}$, then the Hausdorff dimension bound still holds, but for singular points with positive lower densities. See the discussion in Section 2.
\end{remark}

The study of sets of finite perimeter, especially minimizing boundaries, in singular ambient spaces has been carried out in \cite{Ambrosio2019},\cite{Mondino2021},\cite{Kinnunen2013}, focusing on the so called RCD($K,N$) spaces. And recently in \cite{Francesco2023}, the authors establish a monotonicity formula for perimeter minimizing sets in RCD($0,N$) metric measure cones and generalize some classical regularity results to RCD spaces. Compared to these works, our results require nonnegative scalar curvatures instead of Ricci, and therefore can be extended to a large family of metric measure spaces. However, our methods only work for metric spaces that can be isometrically embedded into some larger Euclidean spaces, since the compactness of minimizing currents are based on classical (extrinsic) geometric measure theory (see Lemma \ref{cpt} and \cite{Simon1983}).  

Our next step is to consider non-existence of minimizing hypercones in higher dimensional singular spaces, in order to improve the regularity results above. However, we will show that our bound for dimension in Theorem \ref{main1} is actually optimal by the following concrete example which was first pointed out by Frank Morgan in \cite{Morgan2002}:
\begin{example}\label{main2}
    If $S^n(\lambda)$ is the standard sphere $S^n$ with radius $\lambda \le 1$, $C(S^n(\lambda))$ is the cone with vertex $O$ and cross section $S^n(\lambda)$, $C(S^{n-1}(\lambda))$ is the totally geodesic hypercone in $C(S^n(\lambda))$, then $C(S^{n-1}(\lambda))$ is area-minimizing in $C(S^n(\lambda))$ if and only if $\frac{n\lambda}{2\sqrt{n-1}}\geq 1$.
\end{example}
\begin{remark}
     If $\lambda=1$, then $C(S^n(\lambda))$ is just $\mathbb{R}^{n+1}$. Therefore, this theorem tells us that $C(S^1(\lambda))$ is not area-minimizing in $C(S^2(\lambda))$ for any $\lambda \le 1$ except for the trivial case $\lambda=1$. However, $C(S^{n-1}(\lambda))$ can be area-minimizing in $C(S^n(\lambda))$ if $n\geq 3$ and $\lambda$ is sufficiently close to 1.
\end{remark}

These facts help us understand more about the regularity of area-minimizing hypersurfaces in singular ambient spaces.

Now we briefly explain the main idea of proofs. We start with a conical manifold having an isolated singular point. And then we apply a Hardy-Littlewood-Sobolev type inequality to the stability inequality in the dimension two case (Section 3). By direct computation and detailed analysis, we prove the non-existence of minimizing hypercones in 3-dimensional scalar nonnegative ambient cones with isolated singularities. This result is new in our settings, to the best of our knowledge. Finally, we use standard blow-up arguments (see \cite{Federer1970},\cite{Giusti1984},\cite{Simon1983}) in general cases (Section 4) and result in Theorem \ref{main1}. Note that in the case of singular ambient spaces, a monotonicity formula can only be established for metric cones. Therefore we have to iterate the blow-up procedure in order to obtain local regularity results. And such modification of Federer's dimension reduction argument is similar to the one in \cite{Francesco2023}.

As for Example \ref{main2}, we will give an alternative proof closely related to our discussion of minimizing cones in saclar nonnegative spaces, when $2\le n\le 6$. Indeed, since $S^n$ is rotationally symmetric, we can reduce the problem to an ODE by the moving plane method (see \cite{Alexandrov1962},\cite{Schoen1983}). And the behavior of the ODE (Section 5) gives the exact bound in Example \ref{main2}.





\section{Preliminaries}
In this part, we recall classical results of area-minimizing hypersurfaces in geometric measure theory and extend some of them to singular ambient spaces. And we give the precise definition of $\mathcal{F}=\cup_n\mathcal{F}_n$ ($\mathcal{F}_n$ is the class of admissible spaces mentioned in the introduction) which we are interested in. 

\begin{definition}\label{admissible}
    We define $\mathcal{F}_1$ to be the class of all 2-dimensional smooth manifolds as well as flat cones whose cone angles are less than $2\pi$. And then we say $M\in {\mathcal F}_n(n\geq 2)$ if it satisfies:\\
\vskip 0.1in
\noindent
(1) $M$ is the metric completion of a $(n+1)$-dimensional orientable Riemannian manifold $(M_0, g)$ isometrically embedded in some Euclidean space $\RR^{n+l}$. And the singular set of $M$ has $\mathcal{H}^n$ measure $0$, where $\mathcal{H}^n$ is the $n$-dimensional Hausdorff measure. (See also Definition \ref{singular}.)\\
\vskip 0.1in
\noindent
(2) Existence of tangent cones, that is, for any $p\in M$ and $\{\lambda_i\}\rightarrow 0$, there exists a subsequence $\{\lambda_i\}$ (not labeled) such that $\eta_{p,\lambda_i,\#}M$ converge (in the sense of integral currents) to a metric cone $C$, i.e., $\eta_{0,\lambda,\#}C=C$, $\forall \lambda>0$. In addition, the regular part of $C$ is connected and orientable. Such a $C$ is referred as a tangent cone of $M$ at $p$.\\
\vskip 0.1in
\noindent
(3) For any tangent cone $C$ of $M$, the scalar curvature is non-negative on its regular part, i.e.,
$\text{Scal}_C(q)\geq 0$ for any $q$ outside the singular set of $C$.\\
\vskip 0.1in
\noindent
(4) (Locally $C^1$ convergence to tangent cones) In (2), the regular part of $\eta_{p,\lambda_i,\#}M$ converges to the tangent cone $C$ in the $C^1$ sense locally away from the singular set of $C$.\\
\vskip 0.1in
\noindent
(5) For any tangent cone $C$ of $M$, if $p\in C\backslash \{0\}$ and $\lambda_i\mapsto 0$, then by taking a subsequence if necessary, $\eta_{p,\lambda_i,\#}C$ converge to a tangent cone $C'$ of $C$ at $p$. Then $C'$ is of the form $\RR \times \bar C$, and we require $\bar C\in {\mathcal F}_{n-1}$.\\
\vskip 0.1in
\noindent
(6) The density of $M$ has a uniform upper bound. That is, $$\Theta_M\,:=\sup_{x\in M, r>0}\frac{\mu_M(B_r(x))}{r^{n+1}}<\infty,$$
where $\mu_M$ is the induced measure on $M$.
And at least one of the following holds:\\
\textbf{a.} The singular set $\text{Sing}(C)$ of the tangent cone $C$ of $M$ has its Hausdorff dimension no larger than $n-2$.\\
\textbf{b.} The tangent cone $C$ of $M$ satisfies the following density bound for $p\in\text{Sing}(C)$:
$$\Theta_{C}(p)\,:=\,\overline{\lim_{r \to 0}}\, \frac{\mu(B_r(p))}{\omega_{n+1} r^{n+1}}\,\le\, 1-\epsilon,$$
where $\mu$ is the induced measure on $C$, $\omega_{n+1}$ is the volume of the unit ball in $\RR^{n+1}$ and $\epsilon > 0$ is an uniform constant, particularly, it is independent of $C$ and $p$.\\
\vskip 0.1in
\noindent 
(7) The induced measure $\mu_M$ on $M$ is doubling, and $M$ supports a $(1,1)$-Poincare inequality. (See Definition \ref{doubling}.)\\  
\end{definition}

\subsection{Area-minimizing Currents}
Let $U$ be an open subset of $\mathbb{R}^{k+l}$, $\mathbb{I}_k(U)$ be the set of integral $k$-currents in $U$. And suppose $A\subseteq U$ is any fixed subset, $T\in \mathbb{I}_k(U)$.

\begin{definition}
    We say that $T$ is minimizing in $A$ if
    $$\mathbb{M}_W(T)\le \mathbb{M}_W(S),$$
    whenever $W\subset\subset U$, $\partial S=\partial T$ in $U$, and spt$(S-T)$ is a compact subset of $A\cap W$. Here, $\mathbb{M}_W(T)$ is the mass of $T$ in $W$.
\end{definition}

\begin{remark}
    Classical theory for minimizing currents focuses on the case $A=U$ or $A=N\cap U$ for some smooth manifold $N$ embedded in $\mathbb{R}^{k+l}$. And in the following discussions, we will allow $A$ to be a manifold with singularities.
\end{remark}

\subsection{Riemannian Manifolds with Singularities}
Recall that when $T$ is an integral $k$-current, we define the singular and regular part of $T$ as follows:
$$\text{spt}(T)\setminus\text{spt}(\partial T)=\text{Reg}(T)\cup \text{Sing}(T),$$
$$\text{Reg}(T)=\{x \mid \text{spt}(T) \text{ is a }k\text{-dimensional submanifold in some neighborhood of }x\},$$
$$\text{Sing}(T)=\text{spt}(T)\setminus(\text{spt}(\partial T)\cup\text{Reg}(T)).$$
It is well-known that if $T\in\mathbb{I}_k$ is minimizing in $A$, and $A$ is a smooth $(k+1)$-dimensional manifold, then the Hausdorff dimension of $\text{Sing}(T)$ is no larger than $k-7$. And our main goal is to extend the regularity result to the case when $A$ has singularities.

\begin{definition}\label{singular}
    We say that $(M,g)$ is an $n$-dimensional (orientable) manifold with a singular set $S$, if $M$ is a multiplicity one integral current in some $\mathbb{R}^{n+l}$ with singular set $\text{Sing}(M)=S$, $\mathcal{H}^{n-1}(S)=0$, and $\text{Reg}(M)$ is an $n$-dimensional (orientable) Riemannian manifold endowed with the metric $g$ induced from the Euclidean metric of $\RR^{n+l}$. (Here $M$ is also viewed as a subset of $\mathbb{R}^{n+l}$ by abuse of notation.)
\end{definition}

\subsection{Tangent Cones and Densities}
Let $T\in\mathbb{I}_k$, $x\in$ spt$(T)$, and $\{\lambda_i\}$ be a sequence of positive numbers decreasing to $0$. Then we can define the "blow-up" of $T$ at $x$ as follows:\par
Let $\eta_{x,\lambda_i}(y)=\frac{y-x}{\lambda_i}$ be a dilation of $\mathbb{R}^{k+l}$, and consider a sequence of the push forward currents $(\eta_{x,\lambda_i})_{\#}T$. After passing to a subsequence (not labeled), we obtain:
$$(\eta_{x,\lambda_i})_{\#}T\rightharpoonup C\in\mathbb{I}_k,\quad i\rightarrow\infty,$$
where the convergence is in the weak sense, and $C$ is called the $tangent$ $cone$ of $T$ at $x$. Note that the existence of tangent cones of $T$ is guaranteed by certain compactness theorem for integral currents in classical geometric measure theory, provided $T$ satisfies some boundedness conditions. However, the uniqueness of tangent cones at a given point remains open, even if $T$ is assumed to be area-minimizing. We refer the reader to \cite{Simon1983} for more background knowledge.\par
Another important concept is the density function and the (upper) density of a current $T$ at the point $x\in\text{spt}(T)\setminus\text{spt}(\partial T)$:
$$\Theta_T(x,r)=\Theta_T^k(x,r)=\frac{\mathbb{M}(T\llcorner B_r(x))}{\omega_kr^k},\quad 0<r<\text{dist}(x,\text{spt}(\partial T)).$$
$$\Theta_T(x)=\limsup_{r\rightarrow0}\Theta_T(x,r).$$

\subsection{Compactness of Minimizing Currents}
Recall that an important step in the classical regularity theory is the compactness theorem of minimizing currents. Namely, if $T_j\in\mathbb{I}_k$ are minimizing currents in $A_j$, and $T_j\rightharpoonup T$, $A_j\rightharpoonup A$, as $j\rightarrow\infty$, then $T$ is also minimizing in $A$. However, this compactness result is only proved for $A_j$, $A$ being $C^1$ embedded submanifolds, and $A_j$ $C^1$-converging to $A$. Therefore, we need to extend it to the case when singularities occur in ambient spaces.

We will focus on the multiplicity one class and prove the following compactness lemma for minimizing boundaries:

\begin{lemma}\label{cpt}
    Suppose $A_j\subseteq \mathbb{R}^{k+l}$ is a sequence of $(k+1)$-dimensional manifolds with singularities converging to some $A$ in the current sense. And assume that the convergence is locally $C^1$ away from $S=$Sing$(A)$. In addition, suppose there exists a uniform constant $\Theta$ such that:
    $$\Theta_{A_j}(x,r)\le\Theta,\quad \forall j\in\mathbb{N}_+,\: x\in A_j,\:r>0.$$
    If $T_j=\partial\llbracket E_j\rrbracket$ is a sequence of multiplicity one $k$-currents with spt$(T_j)\subseteq A_j$, each $T_j$ minimizing in $A_j$, and $E_j\subseteq A_j$ are Caccioppoli sets, $i.e.$, $\chi_{E_j}$ has locally bounded variation in $A_j$ ($T_j$ is called a minimizing boundary),
    then there exists $T=\partial\llbracket E\rrbracket\in I_k$ for some Caccioppoli set $E\subseteq A$, spt$(T)\subseteq A$, such that $T_j\rightharpoonup T$, $\mu_{T_j}\rightharpoonup\mu_T$, and $T$ is minimizing in $A$.
\end{lemma}
\begin{proof}

    The existence of the weak limit $T=\partial\llbracket E\rrbracket\in \mathbb{I}_k(A\setminus\text{Sing}(A))$ follows from the standard compactness theorem for minimizing boundaries. We only need to prove that $T$ is also minimizing in $A$. And the proof is a modification of the one in \cite{Simon1983}.
    
    Let $U=B_{2R}$ be an open ball in $\RR^{k+l}$, $K\subseteq B_R$ be an arbitrary compact set and choose a smooth $\varphi$ : $U\rightarrow[0,1]$ such that $\varphi=1$ in some neighborhood of $K$, and spt$(\varphi)\subseteq \{x\in U \mid \text{dist}(x, K)<\eta\}$, where $\eta<\text{dist}(K, \partial U)$ is arbitrary. For $0\le\lambda<1$, let 
    $$W_{\lambda}=\{x\in U \mid \varphi(x)>\lambda\}.$$
    Then $K\subseteq W_{\lambda}\subset\subset U$, for each $0\le\lambda<1$.
    
    Since $\mathcal{H}^{k}(S\cap \overline{B_R})=0$, for every $\varepsilon_0>0$, there exists a covering of $S\cap \overline{B_R}$ which consists of finitely many balls $\{B_{2r_i}(p_i)\}_{i\in I}$ with radius $2r_i<\varepsilon_0$ such that
    $$\sum_{i\in I}\omega_k2^kr_i^k\le \varepsilon_0.$$
    Then by the covering lemma, we can choose a collection of balls $\{B_{2\varepsilon}(p_i')\}_{i\in I'}$ for any $\varepsilon<\min_{i\in I} \{r_i\}$, such that $\{B_{2\varepsilon}(p_i')\}_{i\in I'}$ are disjoint, and $\{B_{4\varepsilon}(p_i')\}_{i\in I'}$ cover $S\cap \overline{B_R}$. Thus, 
    $$\bigcup_{i\in I'}B_{2\varepsilon}(p_i')\subseteq\bigcup_{i\in I}B_{4r_i}(p_i),$$
    and we have:
    $$\sum_{i\in I'}\omega_k2^k\varepsilon^k\le\sum_{i\in I}\omega_k4^kr_i^k\le 2^k\varepsilon_0.$$
    
    Without loss of generality, assume $B_{2\varepsilon}(p_i')\cap S\cap \overline{B_R}\neq \emptyset$, $\forall i\in I'$. Then we can choose $q_i\in B_{2\varepsilon}(p_i')\cap S\cap \overline{B_R}$ such that:
    $$\bigcup_{i\in I'}B_{6\varepsilon}(q_i)\supseteq U_{2\varepsilon}:= \{x\in U\mid \text{dist}(x,S\cap \overline{B_R})<2\varepsilon\}.$$
    Therefore,
    $$\mu_{A_j}(U_{2\varepsilon})\le\sum_{i\in I'}\mu_{A_j}(B_{6\varepsilon}(q_i))\le\Theta\sum_{i\in I'}\omega_{k+1}6^{k+1}\varepsilon^{k+1}\le C\varepsilon_0\varepsilon,$$
    where $C$ is a constant depending only on $k$ and $\Theta$.
    
    Now since $A_j$ converges to $A$ locally away from $S=\text{Sing}(A)$, we can find for $V_{\varepsilon}$ and $A_j$ ($j$ sufficiently large) a diffeomorphism $\psi_{\varepsilon,j}$: $V_{\varepsilon}\rightarrow V_{\varepsilon}$, such that $\psi_{\varepsilon,j}$ gives a diffeomorphism from $A_j\cap V_{\varepsilon}$ onto $A\cap V_{\varepsilon}$, and $\psi_{\varepsilon,j}\rightarrow id_{V_{\varepsilon}}$, as $j\rightarrow\infty$, in the $C^1$ sense. Therefore, since $T_j\rightharpoonup T$, we also have $\psi_{\varepsilon,j,\#}T_j\rightharpoonup T$ in $V_{\varepsilon}$, as $j\rightarrow\infty$. By the equivalence of flat metric convergence and weak convergence, we know that $d_W(T, \psi_{\varepsilon,j,\#}T_j)\rightarrow0$ for each $W\subset\subset V_{\varepsilon}$ ($d_W$ is the flat metric distance), hence in particular,
    $$T-\psi_{\varepsilon,j,\#}T_j=\partial R_{\varepsilon,j},\quad R_{\varepsilon,j}=\llbracket E\rrbracket-\psi_{\varepsilon,j,\#}\llbracket E_j\rrbracket.$$
    And since $\chi_{\psi_{\varepsilon,j}(E_j)}\rightarrow \chi_{E}$ in $L_{loc}^1(V_{\varepsilon})$ by the classical theory for Caccioppoli sets (see \cite{Simon1983}), we have:
    $$\mathbb{M}_W(R_{\varepsilon,j})\rightarrow 0,\quad\forall W\subset\subset V_{\varepsilon}.$$

    From now on, unless otherwise noted, all the currents are considered to be in $V_{\varepsilon}$ rather than $U$.
    
    By the slicing theory, we can choose $0<\alpha<1$ and a subsequence (not labeled) such that
    $$\partial(R_{\varepsilon,j}\llcorner W_{\alpha})=(\partial R_{\varepsilon,j})\llcorner W_{\alpha}+P_{\varepsilon,j},$$
    where spt$(P_{\varepsilon,j})\subseteq \partial W_{\alpha}$ and $\mathbb{M}(P_{\varepsilon,j})\rightarrow0$. Futhermore, we can also choose $\alpha$ such that 
    \begin{equation}\label{alpha}
    \mathbb{M}(T\llcorner \partial W_{\alpha})=0,\quad \mathbb{M}((\psi_{\varepsilon,j,\#}T_j)\llcorner \partial W_{\alpha})=0, \quad \forall j.
    \end{equation}
    Thus, putting these together we obtain:
    $$T\llcorner W_{\alpha}=(\psi_{\varepsilon,j,\#}T_j)\llcorner W_{\alpha}+\partial (R_{\varepsilon,j}\llcorner W_{\alpha})+P_{\varepsilon,j}, \quad \mathbb{M}(R_{\varepsilon,j}\llcorner W_{\alpha})+\mathbb{M}(P_{\varepsilon,j})\rightarrow0.$$\par
    
    Now let $X=\partial\llbracket F\rrbracket\in \mathbb{I}_k(U)$ with $\partial (X-T)=0$ and spt$(X-T)\subseteq K\cap A$. We need to prove $$\mathbb{M}_{W_{\alpha}}(T)\le \mathbb{M}_{W_{\alpha}}(T+X).$$
    By our construction, we have:
    \begin{align*}
        \mathbb{M}_{W_{\beta}}(X)
        &\geq \mathbb{M}_{W_{\beta}\cap V_{\varepsilon}}(X)\\
        &= \mathbb{M}_{W_{\beta}\cap V_{\varepsilon}}(\psi_{\varepsilon,j,\#}T_j+\partial (R_{\varepsilon,j}\llcorner W_{\alpha})+P_{\varepsilon,j}+X-T)\\
        &\geq \mathbb{M}_{W_{\beta}\cap V_{\varepsilon}}(\psi_{\varepsilon,j,\#}T_j+\partial (R_{\varepsilon,j}\llcorner W_{\alpha})+X-T)-\mathbb{M}_{W_{\alpha}\cap V_{\varepsilon}}(P_{\varepsilon,j})\\
        &\geq \delta_{\varepsilon,j}\mathbb{M}_{\psi_{\varepsilon,j}^{-1}(W_{\beta}\cap V_{\varepsilon})}\left((\psi_{\varepsilon,j}^{-1})_{\#}(\psi_{\varepsilon,j,\#}T_j+\partial (R_{\varepsilon,j}\llcorner W_{\alpha})+X-T)\right)-\eta_{\varepsilon,j}\\
        &= \delta_{\varepsilon,j}\mathbb{M}_{\psi_{\varepsilon,j}^{-1}(W_{\beta}\cap V_{\varepsilon})}\left(T_j+\partial(\psi_{\varepsilon,j}^{-1})_{\#} (R_{\varepsilon,j}\llcorner W_{\alpha})+(\psi_{\varepsilon,j}^{-1})_{\#}(X-T)\right)-\eta_{\varepsilon,j},
    \end{align*}
    where $\beta>\alpha$ is arbitrary, and $\delta_{\varepsilon,j}\rightarrow 1$, $\eta_{\varepsilon,j}\rightarrow0$ as $j\rightarrow\infty$. \par
    Note that $$T_j=\partial\llbracket E_j\rrbracket,$$ $$\partial(\psi_{\varepsilon,j}^{-1})_{\#} (R_{\varepsilon,j}\llcorner W_{\alpha})=\partial\llbracket(\psi_{\varepsilon,j}^{-1}(E)\cap\psi_{\varepsilon,j}^{-1}(W_{\alpha}))\rrbracket-\partial\llbracket E_j\cap\psi_{\varepsilon,j}^{-1}(W_{\alpha})\rrbracket,$$ $$(\psi_{\varepsilon,j}^{-1})_{\#}(X-T)=\partial\llbracket(\psi_{\varepsilon,j}^{-1}(F)\cap\psi_{\varepsilon,j}^{-1}(W_{\alpha}))\rrbracket-\partial\llbracket E\cap\psi_{\varepsilon,j}^{-1}(W_{\alpha})\rrbracket,$$
    where we extend the Lipschitz map $\psi_{\varepsilon,j}^{-1}$ to be defined everywhere in $U$. (But the currents under consideration are all in $V_{\varepsilon}$).
    Hence $$T_{\varepsilon,j}:=T_j+\partial(\psi_{\varepsilon,j}^{-1})_{\#} (R_{\varepsilon,j}\llcorner W_{\alpha})+(\psi_{\varepsilon,j}^{-1})_{\#}(X-T)=\partial \llbracket F_{\varepsilon,j}\rrbracket$$ is also a boundary for some $F_{\varepsilon,j}$.\par
    Consider a new current (in $U\setminus (V_{\varepsilon}^c\cap \psi_{\varepsilon,j}^{-1}\overline{W_{\alpha}})$) associated with a set $\llbracket\widetilde{F}_{\varepsilon,j}\rrbracket$:
    $$\llbracket\widetilde{F}_{\varepsilon,j}\rrbracket:=
    \llbracket F_{\varepsilon,j}\rrbracket\llcorner\psi_{\varepsilon,j}^{-1}W_{\beta}+\llbracket E_j\rrbracket\llcorner\psi_{\varepsilon,j}^{-1}W_{\beta}^c.$$
    By the slicing theory, 
    \begin{align*}
        \partial\llbracket\widetilde{F}_{\varepsilon,j}\rrbracket&=\partial(\llbracket F_{\varepsilon,j}\rrbracket\llcorner\psi_{\varepsilon,j}^{-1}W_{\beta})+\partial(\llbracket E_j\rrbracket\llcorner\psi_{\varepsilon,j}^{-1}W_{\beta}^c)\\
        &=T_{\varepsilon,j}\llcorner\psi_{\varepsilon,j}^{-1}W_{\beta}-\langle\llbracket F_{\varepsilon,j}\rrbracket,\varphi\circ\psi_{\varepsilon,j},\beta \rangle + T_j\llcorner\psi_{\varepsilon,j}^{-1}W_{\beta}^c+\langle\llbracket E_j\rrbracket,\varphi\circ\psi_{\varepsilon,j},\beta \rangle\\
        &=T_{\varepsilon,j}\llcorner\psi_{\varepsilon,j}^{-1}W_{\beta}+T_j\llcorner\psi_{\varepsilon,j}^{-1}W_{\beta}^c,
    \end{align*}
    where $\langle T,f,t\rangle$ denotes the slice of a current $T$ by a function $f$ at time $t$ (see \cite{Simon1983}), and in the last equality we use
    $$\langle\llbracket F_{\varepsilon,j}\rrbracket,\varphi\circ\psi_{\varepsilon,j},\beta \rangle=\langle\llbracket E_j\rrbracket,\varphi\circ\psi_{\varepsilon,j},\beta \rangle, $$
    because $F_{\varepsilon,j}$ coincides with $E_j$ outside $\psi_{\varepsilon,j}^{-1}\overline{W_{\alpha}}$.
    
    By the slicing theory again, for almost every $\varepsilon'\in(\varepsilon,2\varepsilon)$,
    $$\langle \partial\llbracket\widetilde{F}_{\varepsilon,j}\rrbracket,r,\varepsilon'\rangle=-\partial\langle\llbracket\widetilde{F}_{\varepsilon,j}\rrbracket,r,\varepsilon'\rangle, $$
    where $r(x)=\text{dist}(x,S\cap \overline{B_R})$. Now if we choose $\varepsilon'$ such that $\psi_{\varepsilon,j}^{-1}W_{\beta}\cap(U\setminus V_{\varepsilon})\subset\subset U_{\varepsilon'}$, and define $$\widetilde{E}_{\varepsilon,j}:=\widetilde{F}_{\varepsilon,j}\cap U_{\varepsilon'}^c,$$
    then
    \begin{align*}
        \partial\llbracket\widetilde{E}_{\varepsilon,j}\rrbracket=\partial(\llbracket\widetilde{F}_{\varepsilon,j}\rrbracket\llcorner U_{\varepsilon'}^c)=(\partial\llbracket\widetilde{F}_{\varepsilon,j}\rrbracket)\llcorner U_{\varepsilon'}^c-\langle\llbracket\widetilde{F}_{\varepsilon,j}\rrbracket,r,\varepsilon' \rangle.
    \end{align*}
    This is a boundary in $A_j\cap U$ which can be viewed as a hypersurface obtained by gluing  $(\partial\llbracket\widetilde{F}_{\varepsilon,j}\rrbracket)\llcorner U_{\varepsilon'}^c$ and $-\langle\llbracket\widetilde{F}_{\varepsilon,j}\rrbracket,r,\varepsilon' \rangle$ along their common boundary $\langle \partial\llbracket\widetilde{F}_{\varepsilon,j}\rrbracket,r,\varepsilon'\rangle$. Since
    $$\langle\llbracket\widetilde{F}_{\varepsilon,j}\rrbracket,r,\varepsilon' \rangle+\langle\llbracket\widetilde{F}_{\varepsilon,j}^c\rrbracket,r,\varepsilon' \rangle=\langle\llbracket U\rrbracket,r,\varepsilon' \rangle,$$
    by choosing $\widetilde{F}_{\varepsilon,j}^c$ instead of $\widetilde{F}_{\varepsilon,j}$ if necessary, we can assume that:
    $$\mathbb{M}(\langle\llbracket\widetilde{F}_{\varepsilon,j}\rrbracket,r,\varepsilon' \rangle)\le\frac{1}{2}\mathbb{M}(\langle\llbracket U\rrbracket,r,\varepsilon' \rangle).$$
    And by the coarea formula as well as the volume estimate $\mu_{A_j}(U_{2\varepsilon})\le C\varepsilon_0\varepsilon$, it is possible to choose $\varepsilon'\in(\varepsilon,2\varepsilon)$ such that:
    $$\mathbb{M}(\langle\llbracket U\rrbracket,r,\varepsilon' \rangle)\le C\varepsilon_0.$$
    Therefore, we have an estimate for the mass of $\partial\llbracket\widetilde{E}_{\varepsilon,j}\rrbracket$ near the singularities:
    $$\mathbb{M}(\langle\llbracket\widetilde{F}_{\varepsilon,j}\rrbracket,r,\varepsilon' \rangle)\le C\varepsilon_0.$$
    Since $T_j=\partial\llbracket E_j\rrbracket$ is a minimizing boundary in $A_j$, and $E_j=\widetilde{F}_{\varepsilon,j}=\widetilde{E}_{\varepsilon,j}$ outside $\psi_{\varepsilon,j}^{-1}\overline{W_{\alpha}}\cup U_{\varepsilon'}$, we have:
    $$\mathbb{M}_{\psi_{\varepsilon,j}^{-1}W_{\beta}\cup U_{\varepsilon'}}(\partial\llbracket \widetilde{E}_{\varepsilon,j}\rrbracket)\geq\mathbb{M}_{\psi_{\varepsilon,j}^{-1}W_{\beta}\cup U_{\varepsilon'}}(T_j)\geq\mathbb{M}_{\psi_{\varepsilon,j}^{-1}W_{\beta}}(T_j).$$
    Thus,
    \begin{align*}
        \mathbb{M}_{\psi_{\varepsilon,j}^{-1}W_{\beta}\cap U_{\varepsilon'}^c}(T_{\varepsilon,j})
        &=\mathbb{M}_{\psi_{\varepsilon,j}^{-1}W_{\beta}\cap U_{\varepsilon'}^c}(\partial\llbracket \widetilde{F}_{\varepsilon,j}\rrbracket)\\
        &=\mathbb{M}_{\psi_{\varepsilon,j}^{-1}W_{\beta}\cap U_{\varepsilon'}^c}(\partial\llbracket \widetilde{E}_{\varepsilon,j}\rrbracket)\\
        &\geq \mathbb{M}_{\psi_{\varepsilon,j}^{-1}W_{\beta}\cap V_{\varepsilon}}(T_j)-\mathbb{M}_{U_{\varepsilon'}(\partial\llbracket\widetilde{E}_{\varepsilon,j}\rrbracket)}\\
        &=\mathbb{M}_{\psi_{\varepsilon,j}^{-1}W_{\beta}\cap V_{\varepsilon}}(T_j)-\mathbb{M}(\langle\llbracket\widetilde{F}_{\varepsilon,j}\rrbracket,r,\varepsilon' \rangle)\\
        &\geq\delta_{\varepsilon,j}\mathbb{M}_{W_{\beta}\cap V_{\varepsilon}}(\psi_{\varepsilon,j,\#}T_j)-C\varepsilon_0.
    \end{align*}
    Combining the inequality:
    $$\mathbb{M}_{W_{\beta}}(X)\geq\delta_{\varepsilon,j}\mathbb{M}_{\psi_{\varepsilon,j}^{-1}(W_{\beta}\cap V_{\varepsilon})}(T_{\varepsilon,j})-\eta_{\varepsilon,j},$$
    we obtain:
    $$\mathbb{M}_{W_{\beta}}(X)\geq\delta_{\varepsilon,j}^2\mathbb{M}_{W_{\beta}\cap V_{\varepsilon}}(\psi_{\varepsilon,j,\#}T_j)-\eta_{\varepsilon,j}-C\varepsilon_0.$$
    Letting $\beta\rightarrow\alpha$, we have:
    $$\mathbb{M}_{W_{\alpha}}(X)\geq\delta_{\varepsilon,j}^2\mathbb{M}_{W_{\alpha}\cap V_{\varepsilon}}(\psi_{\varepsilon,j,\#}T_j)-\eta_{\varepsilon,j}-C\varepsilon_0,$$
    where we use (\ref{alpha}) and the fact that spt$(X-T)\subset\subset W_{\alpha}$. Then we let $j\rightarrow\infty$ and by the lower semi-continuity of mass with respect to weak convergence:
    $$\mathbb{M}_{W_{\alpha}}(X)\geq\mathbb{M}_{W_{\alpha}\cap V_{\varepsilon}}(T)-C\varepsilon_0.$$
    Finally, let $\varepsilon_0\rightarrow0$ (note that $2\varepsilon<\varepsilon_0$), and we obtain the following inequality:
    $$\mathbb{M}_{W_{\alpha}}(X)\geq\mathbb{M}_{W_{\alpha}}(T).$$
    By the definition of $W_{\alpha}$ and the arbitrariness of $R$, $X$, $K$ and $\eta$, we conclude that $T$ is minimizing in $A$ as required.

    It remains to prove that $\mu_{T_j}\rightharpoonup\mu_T$. By choosing $X=T$ in the proof above, we obtain:
    $$\mathbb{M}_{W_{\alpha}}(T)\geq\overline{\lim_{j\to\infty}}\left(\delta_{\varepsilon,j}^2\mathbb{M}_{W_{\alpha}\cap V_{\varepsilon}}(\psi_{\varepsilon,j,\#}T_j)-\eta_{\varepsilon,j}-C\varepsilon_0\right)\geq\overline{\lim_{j\to\infty}}\mathbb{M}_{K\cap V_{\varepsilon}}(T_j)-C\varepsilon_0.$$
    After comparing $T_j$ with the slice $\langle\llbracket E_j\rrbracket, r, \varepsilon'\rangle$ in a similar way, by the minimizing property of $T_j$ and the coarea formula,
    $$\mathbb{M}_{K\cap \overline{V_{\varepsilon}^c}}(T_j)\le\mathbb{M}(\langle\llbracket E_j\rrbracket, r, \varepsilon'\rangle)\le C\varepsilon_0,$$
    for some $\varepsilon<\varepsilon'<2\varepsilon<\varepsilon_0.$ Therefore,
    $$\mathbb{M}_{K\cap V_{\varepsilon}}(T_j)\geq\mathbb{M}_{K}(T_j)-C\varepsilon_0,$$
    and hence,
    $$\mathbb{M}_{W_{\alpha}}(T)\geq\mathbb{M}_{K\cap V_{\varepsilon}}(T_j)\geq\overline{\lim_{j\to\infty}}\mathbb{M}_{K}(T_j)-C\varepsilon_0.$$
    By the arbitrariness of $\alpha$ and $\varepsilon_0$, we have:
    $$\mathbb{M}_{K}(T)\geq\overline{\lim_{j\to\infty}}\mathbb{M}_{K}(T_j).$$
    This together with the lowe semi-continuity of mass gives:
    $$\mu_T(K)=\mathbb{M}_{K}(T)=\lim_{j\to\infty}\mathbb{M}_{K}(T_j)=\lim_{j\to\infty}\mu_{T_j}(K).$$
    Thus, we have $\mu_{T_j}\rightharpoonup\mu_T$ as Radon measures, by the arbitrariness of $K$.
\end{proof}
\begin{remark}
    The general case for integral minimizing currents can be reduced to the minimizing boundary case by a local decomposition of integral currents into boundaries. (See \cite{Simon1983}.)
\end{remark}

\subsection{Doubling measure and (1,1)-Poincare inequality}
\begin{definition}\label{doubling}
    A Borel outer measure $\mu$ on a metric space $(M,d)$ is doubling if there exists a constant $C$ depending on $(M,d)$ such that for every ball $B_r(x)$, we have
    $$0<\mu(B_{2r}(x))\le C\mu_r(x)<\infty.$$
\end{definition}

Given a function $u$ and a non-negative Borel measurable function $g$, we say that $g$ is an upper gradient of $u$ if
$$\lvert u(y)-u(x)\rvert\le\int_{\gamma}gds,$$
whenever $\gamma$ is a rectifiale curve connecting $x$ and $y$.

We say that a metric space $(M,d)$ with a doubling measure $\mu$ supports a $(1,1)$-Poincare inequality, if there are constants $C$ and $\lambda\geq1$ such that, for all $u$ with upper gradient $g_u$ and $B_r$, we have
$$\int_{B_r}\lvert u-\bar{u}_{B_r}\rvert d\mu\le Cr\int_{B_{\lambda r}}g_ud\mu,$$
where
$$\bar{u}_{B_r}=\frac{1}{\mu(B_r)}\int_{B_r}ud\mu.$$

Now given a metric space $M$ with induced measure $\mu$, we can consider the minimizing boundaries in $M$. And the following density bound holds:
\begin{lemma}[\cite{Kinnunen2013} Lemma 5.1]\label{lowerbound}
    Let $T=\partial \llbracket E\rrbracket$ be a minimizing boundary in a metric measure space $M$ with induced measure $\mu$. Suppose $\mu$ is doubling, and $M$ supports a $(1,1)$-Poincare inequality. Then we have
    $$\frac{r\mathbb{M}_{B_r(x)}(T)}{\mu(B_r(x))}\geq\gamma_0>0,\quad\forall x\in \text{spt}(T),\quad r>0,$$
    where $\gamma_0$ depends on the constants in the definition of doubling measure and the Poincare inequality.
\end{lemma}

\subsection{The Hausdorff Dimension of Singular Sets}
In this part, we list some standard properties on the Hausdorff dimension of the singular sets of minimizing hypersurfaces. (See \cite{Giusti1984} for more details.)
\begin{definition}
    Suppose $S\subseteq \RR^d, 0\le k<\infty, 0<\delta\le\infty$. Then we define
    $$\mathcal{H}_{\delta}^k(S):=\omega_k2^{-k}\inf\left\{\sum_{j=1}^{\infty}(\text{diam}S_j)^k\mid S\subseteq\cup_jS_j,\text{ diam}(S_j)<\delta\right\}.$$
    Then $H^k_{\delta}(S)$ is decreasing in $\delta$. And the $k$-dimensional Hausdorff measure is
    $$\mathcal{H}^k(S)=\lim_{\delta\rightarrow0}\mathcal{H}^k_{\delta}(S)=\sup_{\delta}\mathcal{H}^k_{\delta}(S).$$
\end{definition}

We have the following well-known criterion for a set to be Hausdorff measure zero:
\begin{lemma}[\cite{Giusti1984} Lemma 11.2]\label{hm0}
    For every $S\subset\RR^d$, $\mathcal{H}^k(S)=0$ if and only if $\mathcal{H}^k_{\infty}(S)=0$.
\end{lemma}

And the following proposition holds for points in a $k$-dimensional set:
\begin{prop}[\cite{Giusti1984} Proposition 11.3]\label{hkae}
    If $S\subseteq\RR^d$, then for $\mathcal{H}^k$-almost every $x\in S$, we have
    $$\limsup_{r\rightarrow0}\frac{\mathcal{H}^k_{\infty}(S\cap B_r(x))}{\omega_kr^k}\geq2^{-k}.$$
\end{prop}

\begin{lemma}\label{limsing}
     Suppose $A_j\subseteq \mathbb{R}^{k+l}$ is a sequence of $(k+1)$-dimensional manifolds with singularities converging to some $A$ in the current sense. And assume that the convergence is locally $C^1$ away from Sing$(A)$. 
     
     Suppose $T_j=\partial\llbracket E_j\rrbracket$ are minimizing boundaries in $A_j$ converging to $T=\partial\llbracket E\rrbracket$ in $A$ as in Lemma \ref{cpt}. For any $\alpha>0$, define
     $$T^{\alpha}:=\{x\in\text{spt}(T)\mid \Theta_T(x,r)\ge\alpha,\:\forall\: r\in (0,1]\}.$$
     Then for any compact set $K\subseteq\RR^{k+l}$ and any $s>0$, $\alpha>0$,
     $$\mathcal{H}^s_{\infty}(T^{\alpha}\cap\text{Sing}(A)\cap K)\geq\limsup_{j\rightarrow\infty}\mathcal{H}^s_{\infty}(T_j^{\alpha}\cap\text{Sing}(A_j)\cap K).$$
\end{lemma}
\begin{proof}
    For $\varepsilon>0$, let $\{S_j\}$ be a covering of $T^{\alpha}\cap\text{Sing}(A)\cap K$ such that
    $$\mathcal{H}^s_{\infty}(T^{\alpha}\cap\text{Sing}(A)\cap K)>\omega_s2^{-s}\sum_{j}(\text{diam}S_j)^s-\varepsilon.$$
    Without loss of generality, assume $S_j$ are open. Let $W=\cup_jS_j$. We claim that:
    $$T_j^{\alpha}\cap\text{Sing}(A_j)\cap K\subseteq W$$ holds for $j$ large enough. Otherwise, there will exist $x_j\in T_j^{\alpha}\cap\text{Sing}(A_j)\cap K\setminus W$ converging to some $x\in K\setminus W\subseteq K\setminus(T^{\alpha}\cap\text{Sing}(A))$. For each $x_j$ and $r$, choose $r_j<r$ such that $B_{r_j}(x_j)\subseteq B_r(x)$, and $r_j\rightarrow r$ as $j\rightarrow\infty$. Then since $\mu_{T_j}\rightarrow\mu_T$ by Lemma \ref{cpt}, we have:
    $$\Theta_T(x,r)=\lim_{j\rightarrow\infty}\frac{\mu_{T_j}(B_r(x))}{\omega_kr_j^k}\geq\liminf_{j\rightarrow\infty}\frac{\mu_{T_j}(B_{r_j}(x_j))}{\omega_kr_j^k}\geq\alpha,$$
    where in the last inequality we use the definition of $T_j^{\alpha}$ and $x_j\in T_j^{\alpha}$. This shows that $x\in \text{spt}(T)$ and moreover, $x\in T^{\alpha}$. Thus, $x\in K\cap T^{\alpha}\cap \text{Reg}(A)$, which contradicts to $x_j\in \text{Sing}(A_j)\rightarrow x$, since $A_j$ converges to $A$ in the $C^1$ sense locally away from $\text{Sing}(A)$. This proves the claim.
    
    Hence, from the claim we know that:
    $$\mathcal{H}^s_{\infty}(T_j^{\alpha}\cap\text{Sing}(A_j)\cap K)\le\mathcal{H}^s_{\infty}(W)\le\omega_s2^{-s}\sum_{j}(\text{diam}S_j)^s.$$
    Thus,
    $$\limsup_{j\rightarrow\infty}\mathcal{H}^s_{\infty}(T_j^{\alpha}\cap\text{Sing}(A_j)\cap K)\le\mathcal{H}^s_{\infty}(T^{\alpha}\cap\text{Sing}(A)\cap K)+\varepsilon.$$
    And the conclusion follows as $\varepsilon$ is arbitrary. 
\end{proof}

The following proposition is a corollary of Lemma \ref{lowerbound} and shows that we have lower density estimates for area-minimizing currents in admissible spaces. And therefore, combining the proof above, we see that the Hausdorff dimension of singular sets does not decrease when we pass to a limit.

\begin{prop}\label{lowerdensity}
    There exists a constant $\alpha_0>0$, such that whenever $A\in \mathcal{F}$ and $T$ is a minimizing boundary in $A$, it holds that:
    $$\Theta_T(x,r)\geq\alpha_0,\quad\forall x\in\text{spt}(T),\quad r>0.$$
\end{prop}
For a proof of this proposition, see \cite{Kinnunen2013} and Lemma \ref{lowerbound}.

As a consequence, we can obtain the following existence result, which allows us to take limits when considering regularity problems.

\begin{corollary}\label{exist}
    Suppose $A\in \mathcal{F}$ as in, $T$ is a minimizing boundary in $A$, and $\mathcal{H}^k(\text{spt}(T)\cap\text{Sing}(A))>0$. Then there exists a minimizing boundary $T_{\infty}$ in the tangent cone $C$ of $A$ such that $\mathcal{H}^k(\text{spt}(T_{\infty})\cap\text{Sing}(C))>0$.
\end{corollary}
\begin{proof}
    From Lemma \ref{hm0} and Proposition \ref{hkae} it follows that there exists a point $x_0\in \text{spt}(T)\cap\text{Sing}(A)$ and a sequence $\{r_j\}\rightarrow0$ such that:
    $$\mathcal{H}^k_{\infty}(\text{spt}(T)\cap\text{Sing}(A)\cap B_{r_j}(x_0))\geq 2^{-k-1}\omega_k r_j^k.$$
    If we set $A_j:=\eta_{x_0,r_j,\#}A$ and $T_j:=\eta_{x_0,r_j,\#}T$, then from Lemma \ref{cpt} we know that there exists a minimizing boundary $T_{\infty}$ in the tangent cone $C$ of $A$ as the blow-up limit of $T_j$. Thanks to Proposition \ref{lowerdensity}, we have $\text{spt}(T)=\text{spt}(T^{\alpha_0})$ and $\text{spt}(T_{\infty})=\text{spt}(T_{\infty}^{\alpha_0})$. Hence, Lemma \ref{limsing} is applicable and we obtain:
    \begin{align*}
        \mathcal{H}^{k}_{\infty}(\text{spt}(T_{\infty})\cap\text{Sing}(C)\cap B_1)
        &\geq\limsup_{j\rightarrow\infty}\mathcal{H}^{k}_{\infty}(\text{spt}(T_j)\cap\text{Sing}(A_j)\cap B_1)\\
        &=\limsup_{j\rightarrow\infty}r_j^{-k}\mathcal{H}^{k}_{\infty}(\text{spt}(T)\cap\text{Sing}(A)\cap B_{r_j}(x_0))\\
        &\geq 2^{-k-1}\omega_k.
    \end{align*}
    The conclusion follows again from Lemma \ref{hm0}.
\end{proof}

\section{The Dimension Two Case}
Let $(C(M^n),g)$ be a conical manifold with vertex $O$, where $(M^n,g_M)$ is the cross section, and
$$g=dt^2+t^2g_M.$$
And let $\Sigma^{n-1}$ be an embedded minimal hypersurface in $(M,g_M)$. Then we study the possibility whether $C(\Sigma^{n-1})$ can be an $n$-dim area-minimizing hypercone in $C(M^n)$. And we are interested in the case $n=2$ in this section.

We first assume that the scalar curvature of the ambient cone $C(M)$ is strictly positive and prove:
\begin{theorem}\label{cone}
    If $n=2$ and $\text{Scal}_M\geq \lambda^{-2}>1$, $0<\lambda<1$, then the 2-dimensional hypercone $C(\Sigma^1)$ is not stable in the 3-dimensional $C(M^2)$ and therefore not area-minimizing. 
\end{theorem}
\begin{proof}[Proof]
  Direct computation shows that: $C(\Sigma)$ is minimal in $C(M)$ iff $\Sigma$ is minimal in $M$. (See Appendix.) And the stability inequality becomes:
  $$\int_{C(\Sigma)}(\text{Ric}_{C(M)}+\lvert A_{C(\Sigma)}\rvert^2)\eta^2\le\int_{C(\Sigma)}\lvert\nabla{\eta}\rvert^2,\quad \forall \eta \in C_0^{\infty}(C(\Sigma)),$$
  where $A_{C(\Sigma)}$ is the second fundamental form and $\nabla$ is the gradient on $C(\Sigma)$.\\
  Note that when $n=2$,
  $$\text{Ric}_{C(M)}=\frac{\text{Ric}_M-(n-1)}{r^2}=\frac{\text{Scal}_M-1}{r^2}\geq \frac{1-\lambda^2}{\lambda^2r^2},$$ 
  and therefore we obtain:
  $$\int_{C(\Sigma)}\frac{1-\lambda^2}{\lambda^2r^2}\eta^2\le\int_{C(\Sigma)}\lvert\nabla{\eta}\rvert^2,\quad \forall \eta \in C_0^{\infty}(C(\Sigma)).$$
  We will show that this inequality fails when $n=2$ if we choose appropriate $\eta$. \par
  Indeed, we let $\eta(x)=\eta(r)$ depend only on $r=\text{dist}(x,O).$ Define
  \begin{equation*}
      \eta(r)=\left\{
    \begin{aligned}
        &\epsilon^{-1}r,& 0&\le r\le \epsilon,\\
        &1,&\epsilon&< r \le R,\\
        &2-R^{-1}r,&  R&<r.  
    \end{aligned}
    \right
    .
  \end{equation*}
  And the stablility inequality above for the test function $\eta$ becomes:
  $$C(n,\lambda)\int_0^{+\infty}\frac{\eta^2}{r^2}Vol(\Sigma_r)dr\le\int_0^{+\infty}(\eta')^2Vol(\Sigma_r)dr.$$
  When $n=2$, $Vol(\Sigma_r)=Cr$, and therefore we obtain:
  $$C(\lambda)\int_{\epsilon}^{R}\frac{1}{r}dr\le\frac{1}{\epsilon^2}\int_0^{\epsilon}rdr+\frac{1}{R^2}\int_R^{2R}rdr.$$
  Letting $\epsilon\rightarrow0$ and $R\rightarrow+\infty$, we get a contradiction. Hence $C(\Sigma^1)$ is unstable in $C(M^2)$ in the dimension 2 case.
\end{proof}

Now we want to extend the results to $\text{Scal}\geq 0$ ambient spaces.

\begin{theorem}\label{thm3.2}
    If $n=2$, $Scal_{C(M)}\geq 0$, $Scal_{C(M)}(p)>0$ at some point $p$, $M$ is complete, connected and oriented, then the 2-dimensional hypercone $C(\Sigma^1)$ can never be area-minimizing in the 3-dimensional $C(M^2)$.
\end{theorem}
\begin{proof}
    Note that the conditions are equivalent to $\text{Scal}_M\geq 1$ and $\text{Scal}_M(q)>1$ for some $q\in M$. Therefore, using the Gauss-Bonnet theorem, we know that $M$ is just $S^2$. And under some diffeomorphism, we can assume that $M=S^2$, and $\Sigma=S^1$ is the equator, without loss of generality. Note that here $M=S^2$ is equipped with a metric $g_M$ different from the standard metric.\par
    First observe that if such a minimizing cone $C(\Sigma)$ exists, then $\Sigma$ must be a geodesic on $M$, by the minimal property. And $\text{Scal}_{C(M)}$ restricted to $C(\Sigma)$ must vanish identically. Otherwise, the same argument in the proof of Theorem \ref{cone} is applicable and we will get a contradiction to the stability inequality. Therefore, we will assume these additional conditions from now on.\par
    By the volume comparison theorem (see \cite{Petersen2018}), since $\text{Scal}_M=\text{Ric}_M$ when $n=2$, we know that in the geodesic polar coordinate system, the volume form satisfies: 
    $$dvol_{g_M}\le dvol_{g_{S^2}}.$$
    So our strategy is to find a surface with boundary $\Sigma$ and smaller area than the cone $C(\Sigma)$, and it is natural to expect that such a surface exists. 
    
    Indeed, we will construct a disk $D_{\delta,\alpha}$ with strictly smaller area than the same Euclidean disk and connect $\partial D_{\delta,\alpha}$ with $\Sigma$ by a catenoid. By an observation that $\mathbb{R}^2$ is not strictly minimizing in $\mathbb{R}^3$, which was first discovered by Hardt and Simon in \cite{Hardt1985}, as well as the theorem above in the strictly positive scalar case, we can find a surface with smaller area when $\delta$ and $\alpha$ are sufficiently close to $0$, as desired. And we carry out the details in the following paragraphs.\par
    Let $(\theta,t)$ be coordinates on $M=S^2$, where $t$ represents the distance from $\Sigma=S^1\subseteq S^2=M$ under the metric $g_M$. And let $(r,\theta,t)$ be the spherical coordinates of $C(M)=C(S^2)=\mathbb{R}^3$, where $r$ is the distance from $O$. And consider the graph of a Lipschitz rotationally symmetric function $f$ (independent of $\theta$) over the upper hemisphere:
    $$G_f=\{(r,\theta,t)\in\mathbb{R}_+^3\mid r=f(t)\}.$$
    Using the coarea formula, we can compute:
    \begin{equation*}
        \text{Area}_{g_{C(M)}}(G_f)=\int_{0}^{+\infty}f(t)L(t)\frac{1}{\lvert\nabla^{C(M)}t\rvert}dt,
    \end{equation*}
    where 
    $$L(t)=\text{Length}\left(\{q\in M\mid t(q)=\text{dist}_{g_M}(q,\Sigma)=t\}\right).$$
    And since $t$ is a distance function on $M$, by direct computation, 
    $$\frac{1}{\lvert\nabla^{C(M)}t\rvert}=\sqrt{f'(t)^2+f(t)^2\frac{1}{\lvert\nabla^{M}t\rvert^2}}=\sqrt{f'(t)^2+f(t)^2}.$$
    Therefore,
    \begin{equation}\label{eq3.1}
    \text{Area}_{g_{C(M)}}(G_f)=\int_{0}^{+\infty}f(t)L(t)\sqrt{f'(t)^2+f(t)^2}dt.
    \end{equation}
    Note that in particular, if $g_{C(M)}$ is the Euclidean metric on $\mathbb{R}^3$ and $g_M$ is the standard metric on $S^2$,
    then
    \begin{equation*}
    L(t)=\left\{
    \begin{aligned}
        &2\pi\cos t,\quad &0\le t\le\frac{\pi}{2},\\
        &0,\quad &t>\frac{\pi}{2}.
    \end{aligned}
    \right
    .
    \end{equation*}
    Now the question is clear. We need to find an $f$ with $f(0)=1$ such that:
    \begin{equation}\label{eq3.2}
        \text{Area}(G_f)=\int_{0}^{+\infty}f(t)L(t)\sqrt{f'(t)^2+f(t)^2}dt<\frac{1}{2}L(0)=\text{Area}(C(\Sigma)).
    \end{equation}

    Actually, we will construct an $f$ satisfying:
    \begin{equation}\label{cat}
        f\cos t=a\cosh\left(\frac{f\sin t-b}{a}\right),\quad 0\le t\le\delta,
    \end{equation}
    \begin{equation}\label{disk}
        (f')^2+f^2=\frac{L(0)^2}{L^2}(f')^2,\quad\quad t\geq\delta,
    \end{equation}
    \begin{equation}\label{eq3.5}
        f(0)=1,\quad\quad f(\delta)=\alpha,
    \end{equation}
    where $a$, $b$ are constants chosen to satisfy (\ref{eq3.5}), and $\delta$, $\alpha$ are sufficiently small. Thus, $G_f$ will be the union of a catenoid and a disk. And we will show that compared to the cone $C(\Sigma)$, the catenoid has $o(\alpha^2\delta^2)$ larger area, while the disk has $O(\alpha^2\delta^2)$ smaller area. And therefore we are done.
    
    We divide the remaining proof into several lemmas below.

    Denote the catenoid in (\ref{cat}) by $G_{\delta,\alpha}$, and the disk in (\ref{disk}) by $D_{\delta,\alpha}$. Then we have:

    \begin{lemma}\label{lem3.3}
        \begin{equation*}
            \text{Area}(G_{\delta,\alpha})=\frac{1}{2}L(0)\left(1-\alpha^2+\alpha^2\delta^2+\frac{1}{-\ln\alpha}\alpha^2\delta^2+o(C_1\delta^2)\right),
        \end{equation*}
        where $C_1=C_1(\alpha)$ is a constant depending only on $\alpha$. 
    \end{lemma}

    \begin{lemma}\label{lem3.4}
        \begin{equation*}
            \text{Area}(D_{\delta,\alpha})\le\frac{1}{2}L(0)\left(\alpha^2-\frac{L(0)^2}{F^2}\alpha^2\delta^2+o(C_2\delta^2)\right),
        \end{equation*}
        where $C_2=C_2(\alpha)$ is a constant depending only on $\alpha$, and $F=\text{Area}_{g_M}(S_+^2)$ is the area of the disk bounded by $\Sigma$ on $S^2$.
    \end{lemma}

    \begin{lemma}[Isoperimetric Inequality \cite{Ionin1969}, \cite{Burago1973}]\label{isoperimetric}
        If $\Sigma$ is a closed curve and $F$ is a disk bounded by $\Sigma$ (whose area is also denoted by $F$) on a closed 2-dimensional Riemannian manifold $M$. Then
        $$L^2+2F\int_F(K_M-\lambda)_+-4\pi\chi F+\lambda F^2\geq0,\quad \forall \lambda\in\mathbb{R},$$
        where $L$ is the length of $\Sigma$, $K_M$ is the Gaussian curvature of $M$, and $\chi$ is the Euler characteristic of the domain $F$. (Here, $u_+:=\max\{u,0\}$ for any function $u$.)\par 
        In particular if we set $\lambda=1$, we will get $L(0)\geq F$ in Lemma \ref{lem3.4}.
        And actually the strict inequality $L(0)>F$ holds.
    \end{lemma}

    Combining Lemma \ref{lem3.3}, \ref{lem3.4} and \ref{isoperimetric}, we see that if we first choose $\alpha$ small enough and then let $\delta$ be sufficiently close to $0$, we will obtain a graph $G_f=G_{\delta,\alpha}\cup D_{\delta,\alpha}$ satisfying (\ref{eq3.2}). Therefore, we conclude that the cone $C(\Sigma)$ is not area-minimizing in $C(M)$, completing the proof.     
\end{proof}

It remains to check these lemmas. And we prove Lemma \ref{lem3.3} at first.

\begin{proof}[Proof of Lemma \ref{lem3.3}]
    Let us first consider the special case: 
    $$L(t)=L(0)\cos t, \quad t\in[0,\delta].$$
    Then the catenoid $G_{\delta,\alpha}$ defined in (\ref{cat}) is a minimizer of the area functional (\ref{eq3.1}), since everything is similar to the standard Euclidean space. One can check this fact by direct computation (but it seems not to be an easy task), or by changing the polar coordinates into the Cartesian coordinates and checking that the catenoid satisfies the minimal surface equation for graphs over $\mathbb{R}^2$.

    Therefore, $f$ satisfies the Euler-Lagrange equation:
    \begin{equation}
        \left(\frac{f'\cos{t}}{\sqrt{(f')^2+f^2}}\right)'-\frac{2f\cos{t}}{\sqrt{(f')^2+f^2}}=0.
    \end{equation}
    Hence, by virtue of (\ref{eq3.1}) and (\ref{eq3.5}), we obtain:
    \begin{align*}
        \text{Area}(G_{\delta,\alpha})&=L(0)\int_0^{\frac{\pi}{2}}\sqrt{(f')^2+f^2}f\cos t dt\\
        &=L(0)\int_0^{\delta}\left((f')^2+f^2\right)\left(\frac{f\cos t}{\sqrt{(f')^2+f^2}} \right) dt\\
        &=L(0)\int_0^{\delta}\frac{(f')^2f\cos t}{\sqrt{(f')^2+f^2}}+f^2\left(\frac{f\cos t}{\sqrt{(f')^2+f^2}} \right) dt\\
        &=\frac{1}{2}L(0)\int_0^{\delta}\frac{(2f'f)f'\cos t}{\sqrt{(f')^2+f^2}}+f^2\left(\frac{f'\cos t}{\sqrt{(f')^2+f^2}}\right)'dt\\
        &=\frac{1}{2}L(0)\left(\frac{f'f^2\cos t}{\sqrt{(f')^2+f^2}}\right)\Bigg|_{0}^{\delta}\\
        &=\frac{1}{2}L(0)\left(\frac{-f'(0)}{\sqrt{1+f'(0)^2}}-\frac{-f'(\delta)\alpha^2\cos\delta}{\sqrt{\alpha^2+f'(\delta)^2}}\right),
    \end{align*}
    Now by (\ref{cat}) and (\ref{eq3.5}), $a$, $b$ is uniquely determined by the equations:
    \begin{equation}\label{eq3.7}
        1=a\cosh\frac{-b}{a},
    \end{equation}
    \begin{equation}\label{eq3.8}
        \alpha\cos\delta=a\cosh\frac{\alpha\sin\delta-b}{a}.
    \end{equation}
    And by differentiating (\ref{cat}), we obtain the equations:
    \begin{equation}\label{eq3.9}
        f'(0)=\sinh\frac{-b}{a},
    \end{equation}
    \begin{equation}\label{eq3.10}
        f'(\delta)\cos\delta-\alpha\sin\delta=(f'(\delta)\sin\delta+\alpha\cos\delta)\sinh\frac{\alpha\sin\delta-b}{a}.
    \end{equation}
    It is convenient to introduce:
    $$-\cos H_{\delta}=\frac{f'(\delta)}{\sqrt{\alpha^2+f'(\delta)^2}},\quad \sin H_{\delta}=\frac{f(\delta)}{\sqrt{\alpha^2+f'(\delta)^2}}.$$
    Then (\ref{eq3.10}) becomes:
    \begin{equation}\label{eq3.11}
        -\cot (H_{\delta}-\delta)=\sinh\frac{\alpha\sin\delta-b}{a}.
    \end{equation}
    By virtue of (\ref{eq3.8}), this gives:
    \begin{equation}\label{eq3.12}
        H_{\delta}=\delta+\arcsin\frac{a}{\alpha\cos\delta}.
    \end{equation}
    Substituting (\ref{eq3.7}), (\ref{eq3.9}) and (\ref{eq3.12}) into the formula for $\text{Area}(G_{\delta,\alpha})$, we obtain:
    \begin{equation}\label{eq3.13}
        \text{Area}(G_{\delta,\alpha})=\frac{1}{2}L(0)\left(\sqrt{1-a^2}-\alpha^2\cos\delta\cos\left(\delta+\arcsin\frac{a}{\alpha\cos\delta}\right)\right).
    \end{equation}
    
    Now we try to represent $a$ in terms of $\delta$ and $\alpha$. Again, by (\ref{eq3.7}),
    $$e^\frac{b}{a}=\frac{1}{a}+\sqrt{\frac{1}{a^2}-1}.$$
    After substituting this into (\ref{eq3.8}), we obtain:
    \begin{align*}
        \alpha\cos\delta&=\frac{1}{2}a\left(e^{\frac{\alpha\sin\delta}{a}}e^{-\frac{b}{a}}+e^{-\frac{\alpha\sin\delta}{a}}e^{\frac{b}{a}}\right)\\
        &=\frac{1}{2}\left(\frac{a^2e^{\frac{\alpha\sin\delta}{a}}}{1+\sqrt{1-a^2}}+e^{-\frac{\alpha\sin\delta}{a}}(1+\sqrt{1-a^2})\right).
    \end{align*}
    Now if we fix $\alpha\in(0,1)$, then standard arguments in analysis show that: 
    $$\lim_{\delta\rightarrow0}a=0,
    \quad\lim_{\delta\rightarrow0}\frac{a}{\delta}=\frac{\alpha}{-\ln\alpha}.$$
    Equivalently,
    $$a=\frac{\alpha}{-\ln\alpha}\delta+o(\delta).$$
    Consequently, expand (\ref{eq3.13}) and we have:
    \begin{align*}
        \text{Area}(G_{\delta,\alpha})=&~\frac{1}{2}L(0)\Bigg[1-\frac{\alpha^2}{2\ln^2\alpha}\delta^2\\&~\phantom{\frac{1}{2}L(0)\Bigg[}-\alpha^2\left(1-\frac{1}{2}\delta^2\right)\left(1-\frac{1}{2}\delta^2\left(1+\frac{\alpha}{-\ln\alpha}\right)^2\right)+o(C_1\delta^2)\Bigg]\\
        =&~\frac{1}{2}L(0)\left(1-\alpha^2+\alpha^2\delta^2+\frac{1}{-\ln\alpha}\alpha^2\delta^2+o(C_1\delta^2)\right),
    \end{align*}
    where $C_1=C_1(\alpha)$ is a constant depending only on $\alpha$. This completes the proof in case $L(t)=L(0)\cos t.$

    Finally, for general $L(t)$, we have by the first and second variation formula (note that $L(t)$ is smooth in a neighborhood of $0$),
    $$L'(0)=0,\quad L''(0)=-L(0),$$
    since $\Sigma$ is a geodesic in $M$ and $\text{Scal}_M=1$ on $\Sigma$. Hence, we obtain:
    $$L(t)-L(0)\cos t=o(t^2).$$
    And consequently,
    \begin{align*}
        \text{Area}(G_{\delta,\alpha})&=\int_{0}^{\delta}f(t)\left(L(t)-L(0)\cos t+L(0)\cos t\right)\sqrt{f'(t)^2+f(t)^2}dt\\
        &=\frac{1}{2}L(0)\left(1-\alpha^2+\alpha^2\delta^2+\frac{1}{-\ln\alpha}\alpha^2\delta^2+o(C_1\delta^2)\right),
    \end{align*}
    where $C_1=C_1(\alpha)$ is a constant depending only on $\alpha$, completing the proof.    
\end{proof}

\begin{remark}
    In \cite{Hardt1985}, Hardt and Simon introduced the concept of \textit{strictly minimizing cones}. A cone $C^n$ is called \textit{strictly minimizing} in $\mathbb{R}^{n+1}$, if there exists $\Theta>0$ such that:
    $$\text{Area}(C\cap B_1)\le\text{Area}(S)-\Theta\epsilon^n,$$
    whenever $\epsilon>0$ and $S\subseteq\mathbb{R}^{n+1}\setminus B_{\epsilon}$ with $\partial S=\partial(C\cap B_1)$. And they pointed out that $\mathbb{R}^n$ is strictly minimizing in $\mathbb{R}^{n+1}$ if and only if $n\geq3$. Actually, we have confirmed this fact by direct computation above, when $n=2$. (For $n\geq 3$, see the discussion in Section 5.) And we find it interesting that this relates closely to the behavior of minimizing cones with isolated singularity.
\end{remark}

Now we turn to the proof of Lemma \ref{lem3.4}.
\begin{proof}[Proof of Lemma \ref{lem3.4}]
    By (\ref{disk}) and (\ref{eq3.5}), $f$ has an explicit form:
    \begin{equation}\label{eq3.14}
        f(t)=\alpha e^{-g_{\delta}(t)},\quad t\geq \delta,
    \end{equation}
    \begin{equation}\label{eq3.15}
        g_{\delta}(t)=\int_{\delta}^{t}\frac{L(s)ds}{\sqrt{L(0)^2-L(s)^2}}.
    \end{equation}
    And by virtue of (\ref{disk}), we have:
    \begin{align*}
        \text{Area}(D_{\delta,\alpha})&=\int_{\delta}^{+\infty}f(s)L(s)\sqrt{f'(s)^2+f(s)^2}ds\\
        &=\int_{\delta}^{+\infty}-L(0)f(s)f'(s)ds\\
        &=\frac{1}{2}L(0)\left(f(\delta)^2-\lim_{t\rightarrow+\infty}f(t)^2\right)\\
        &=\frac{1}{2}L(0)\left(\alpha^2-\lim_{t\rightarrow+\infty}f(t)^2\right).
    \end{align*}
    Now we define $F(t)=\text{Area}(\{q\in M\mid t(q)=\text{dist}_{g_M}(q,\Sigma)\le t\})$, and we have the following properties for $F(t)$ and $L(t)$, which are mentioned in \cite{Burago1973},\\
    (1) $F(t)$ is absolutely continuous, has one-sided derivatives everywhere, and
    $$F'(t)=L(t).$$
    (2) By the formula (12) in \cite{Burago1973} ($A=0$, $K=1$),
    $$F'(t)^2+F(t)^2\le F'(0)^2=L(0)^2.$$
    Therefore, we obtain in (\ref{eq3.15}),
    $$g_{\delta}(t)=\int_{\delta}^{t}\frac{L(s)ds}{\sqrt{L(0)^2-L(s)^2}}\le\int_{\delta}^{t}\frac{F'(s)ds}{F(s)}=\ln\frac{F(t)}{F(\delta)}.$$
    Note that $$F(\delta)=F(0)+F'(0)\delta+o(\delta)=L(0)\delta+o(\delta).$$
    Combining all these formulae, we obtain:
    \begin{align*}
        \text{Area}(D_{\delta,\alpha})
        &=\frac{1}{2}L(0)\left(\alpha^2-\lim_{t\rightarrow+\infty}\alpha^2e^{-2g_{\delta}(t)}\right)\\
        &\le\frac{1}{2}L(0)\left(\alpha^2-\lim_{t\rightarrow+\infty}\alpha^2\frac{F(\delta)^2}{F(t)^2}\right)\\
        &=\frac{1}{2}L(0)\left(\alpha^2-\frac{L(0)^2}{F^2}\alpha^2\delta^2+o(C_2\delta^2)\right),
    \end{align*}
    where $C_2=C_2(\alpha)$ is a constant depending only on $\alpha$, and $F=\text{Area}_{g_M}(S_+^2)$ is the area of the disk bounded by $\Sigma$ on $S^2$. This completes the proof of the lemma.
\end{proof}

Finally, we prove the isoperimetric inequality on $M$.
\begin{proof}[Proof of Lemma \ref{isoperimetric}]
    We refer the reader to \cite{Ionin1969} and \cite{Burago1973} for a proof of the inequality. And we only check that the strict inequality $L(0)>F$ holds in our settings.\par
    Indeed, we choose $F$ to be the disk $S_+^2$ bounded by $\Sigma=S^1$ on $M=S^2$. And if we let $\lambda=1$, combining the facts that $K_M\geq 1$, $\chi=1$, $\Sigma$ is a geodesic of $M$ (since $C(\Sigma)$ is minimal), as well as the Gauss-Bonnet theorem, we will get:
    $$0\le L(0)^2+2F\int_F(K_M-1)_+-4\pi\chi F+F^2=L(0)^2-F^2.$$
    If the equality holds, then equality holds in (5) and (7) in \cite{Burago1973}. And using the Gauss-Bonnet theorem again, this gives:
    $$\int_{F\setminus F(t)}(K_M-1)_+=\int_{F}(K_M-1)_+\text{ holds for all }t\geq0,$$
    which implies $K_M\equiv 1$ in $F=S_+^2$, a contradiction.
\end{proof}
In conclusion, we have finished the proof of Theorem \ref{thm3.2}.

Once we establish the non-existence of minimizing cone, we can apply the standard blow-up argument to manifolds with isolated singular points to obtain:

\begin{corollary}\label{cor}
    Let $M^3$ be a 3-dim manifold with isolated singularity set $S$. Suppose the tangent cone $C_p$ of $M^3$ at each singular point $p$ has the conical metric described at the beginning of this section. And suppose $\eta_{p,\lambda_i,\#}M^3$ converges to its tangent cone $C_p$ in the $C^1$ sense locally away from $p$. If $C_p$ has connected and orientable cross section, and $\text{Scal}_{C_p}\geq0$ holds for all these tangent cones $C_p$, then all the area-minimizing 2-dim hypersurfaces $\Sigma^2$ in $M^3$ are smooth.
\end{corollary}
\begin{proof}
    Suppose on the contrary that $\Sigma^2$ is area-minimizing in $M^3$, but not smooth. Then by the classical regularity theory of minimal surfaces (see \cite{Giusti1984} or \cite{Simon1983} for a review), we know that $\Sigma^2 \cap S \neq \emptyset$. Otherwise $\Sigma^2$ will be a singular area-minimizing hypersurface in the smooth manifold $\overline{M^3}\setminus S$, which is impossible. Now choose $p\in \Sigma^2 \cap S$ and blow up $M^3$ at $p$. By Lemma \ref{cpt}, as well as the monotonicity formula in the conical case which will be established in the next section, we will get an area-minimizing cone in the tangent cone $C_p$ of $M^3$ at $p$. Since $\text{Scal}_{C_p}\geq0$, the cross section $N$ of $C_p$ (closed manifold) has $\text{Scal}\geq1$. If the curvature of the cross section $N$ is identically 1, then by the classification of constant curvature spaces (see \cite{DoCarmo1992} Chapter 8, Proposition 4.4), we know that $N$ is the standard $S^2$. Thus $C_p$ is just the Euclidean $\mathbb{R}^3$. In this case, obviously $p$ cannot be a singular point. Hence $\text{Scal}_N(q)>1$ holds for some $q\in N$. And then Theorem \ref{thm3.2} applies and we get a contradiction.
\end{proof}

\section{Regularity Results}
In this section, we extend the regularity results obtained in Section 3 to general dimensions and carry out the details of blow-up arguments in the proof of Theorem \ref{main1}. 

We first extend the usual monotonicity formula for minimal hypersurfaces to the case when the ambient space is a cone. This is essentially the same as Theorem 3.7 in \cite{Yau2017}.

\begin{lemma}\label{mono}
   If $\Sigma^n$ is an area-minimizing hypersurface in a cone $C^{n+1}$ centered at $0$, then the usual monotonicity formula for $\Sigma^n$ at $0$ still holds. That is, for all $0<s<t$, 
   $$\frac{Vol(B_t\cap \Sigma)}{t^n}-\frac{Vol(B_s\cap \Sigma)}{s^n}=\int_{(B_t\setminus B_s)\cap \Sigma}\lvert x \rvert^{-n-2}\lvert x^{\bot}\rvert^2 dVol(x),$$
   where $x^{\bot}$ is the component of the position vector $x$ perpendicular to $\Sigma^n$ in $C^{n+1}$. 
\end{lemma}
\begin{proof}
    We only need to verify that $\text{div}_{\Sigma}X=n$ still holds in this case, where $X$ is the position vector field. And then the remaining part is the same as the classical proof of the monotonicity formula. (See \cite{Simon1983}. And note that $X$ belongs to the tangent space of the cone $C$.)\par 
    To see this, let $\{e_i\}_{i=1}^{n}$ be an orthonormal  of $\Sigma$, then
    $$\text{div}_{\Sigma}X=\langle \nabla_{e_i}^{\Sigma}X, e_i \rangle=\langle \nabla_{e_i}^{C}X, e_i \rangle=\langle \nabla_{e_i}X, e_i \rangle=n,$$
    where the last gradient is taken in the ambient Euclidean space $\mathbb{R}^{n+l}$.\par
    
    In the following discussion, we will assume $0$ is an isolated singular point of the cone $C$, for simplicity. And the general case follows by multiplying an appropriate cut-off function near singularities in a similar way.
    
    Define $Y=\gamma(r)X$, where $r(x)=\text{dist}(x,0)$, $\gamma:\mathbb{R}\rightarrow[0,1]$ is a cut-off function to be determined. Then,
    \begin{equation}\label{eq4.1}
    \begin{aligned}
        \text{div}_{\Sigma}Y&=\gamma(r)\text{div}_{\Sigma}X+\gamma'(r)\langle \nabla^{\Sigma}r, X \rangle\\
        &=n\gamma(r)+r\gamma'(r)\langle \nabla^{\Sigma}r, \nabla r \rangle\\
        &=n\gamma(r)+r\gamma'(r)\langle \nabla^{\Sigma}r, \nabla^{\Sigma} r \rangle\\
        &=n\gamma(r)+r\gamma'(r)\left(1-\frac{\lvert X^{\bot}\rvert^2}{\lvert X\rvert^2}\right).
    \end{aligned}
    \end{equation}
    Now take $\delta,\sigma,\epsilon\in(0,s/2)$ and a $C^1$ function $\varphi:\mathbb{R}\rightarrow[0,1]$ supported on $[\delta,1+\epsilon]$, with $\varphi\equiv1$ on $[\delta+\sigma,1]$, $0\le\varphi'\le2/\sigma$ on $[\delta,\delta+\sigma]$, $\varphi'\le0$ on $[1,1+\epsilon]$. And then let $\gamma(r)=\varphi(r/t)$. Since
    $$-t\frac{\partial}{\partial t}\left(\varphi\left(\frac{r}{t}\right)\right)=\frac{r}{t}\varphi'\left(\frac{r}{t}\right)=r\gamma'(r),$$
    integrating (\ref{eq4.1}), we obtain:
    $$n\int_{\Sigma}\varphi(r/t)d\mu-t\frac{d}{dt}\left(\int_{\Sigma}\varphi(r/t)d\mu\right)+t\int_{\Sigma}\frac{\partial}{\partial t}\varphi(r/t)\frac{\lvert X^{\bot}\rvert^2}{\lvert X\rvert^2}d\mu=0,$$
    where we use $\int_{\Sigma}\text{div}_{\Sigma}Y=0$, because $\Sigma$ is minimal. (Note that $Y$ is a tangent vector field supported on a compact set of the regular part of $\Sigma$.) Thus, multiplying by $t^{-n-1}$, we have:
    $$\frac{d}{dt}\left(t^{-n}\int_{\Sigma}\varphi(r/t)d\mu\right)=t^{-n}\int_{\Sigma}\frac{\partial}{\partial t}\varphi(r/t)\frac{\lvert X^{\bot}\rvert^2}{\lvert X\rvert^2}d\mu.$$
    Since 
    $$(1+\epsilon)^{-n}r^{-n}\le t^{-n}\le r^{-n},\quad\text{when }\varphi'(r/t)\le0;$$
    $$\delta^nr^{-n}\le t^{-n}\le(\delta+\sigma)^n r^{-n},\quad\text{when }\varphi'(r/t)\geq0,$$
    we have:
    \begin{align*}
    (1+\epsilon)^{-n}\int_{\Sigma}r^{-n}\frac{\partial}{\partial t}\varphi(r/t)\frac{\lvert X^{\bot}\rvert^2}{\lvert X\rvert^2}d\mu&\le\frac{d}{dt}\left(t^{-n}\int_{\Sigma}\varphi(r/t)d\mu\right)\\
    &\le \int_{\Sigma}r^{-n}\frac{\partial}{\partial t}\varphi(r/t)\frac{\lvert X^{\bot}\rvert^2}{\lvert X\rvert^2}d\mu+C(t)(\delta+\sigma)^n,
    \end{align*}
    where $C(t)$ is a constant depending only on $t$.
    By integration with respect to $t$ from $s$ to $t_0$, we obtain:
    \begin{align*}
        &(1+\epsilon)^{-n}\int_{\Sigma}r^{-n}(\varphi(r/t_0)-\varphi(r/s))\frac{\lvert X^{\bot}\rvert^2}{\lvert X\rvert^2}d\mu\\
        \le &~t_0^{-n}\int_{\Sigma}\varphi(r/t_0)d\mu-s^{-n}\int_{\Sigma}\varphi(r/s)d\mu\\
        \le &\int_{\Sigma}r^{-n}(\varphi(r/t_0)-\varphi(r/s))\frac{\lvert X^{\bot}\rvert^2}{\lvert X\rvert^2}d\mu+C(\delta+\sigma)^n.
    \end{align*}
    Letting $\sigma,\epsilon\rightarrow0$, we obtain:
    \begin{align*}
        \int_{(B_{t_0}\setminus B_s)\cap \Sigma}\lvert x \rvert^{-n-2}\lvert x^{\bot}\rvert^2 dVol(x)&\le\frac{Vol((B_{t_0}\setminus B_{\delta})\cap \Sigma)}{t_0^n}-\frac{Vol(B_s\setminus B_{\delta})\cap \Sigma)}{s^n}\\
        &\le\int_{(B_{t_0}\setminus B_s)\cap \Sigma}\lvert x \rvert^{-n-2}\lvert x^{\bot}\rvert^2 dVol(x)+C\delta^2,
    \end{align*}
    Finally, letting $\delta\rightarrow0$, we get the desired monotonicity formula:
    $$\frac{Vol(B_{t_0}\cap \Sigma)}{t_0^n}-\frac{Vol(B_s\cap \Sigma)}{s^n}=\int_{(B_{t_0}\setminus B_s)\cap \Sigma}\lvert x \rvert^{-n-2}\lvert x^{\bot}\rvert^2 dVol(x),\quad 0<s<t_0.$$
\end{proof}

\begin{corollary}\label{cor2}
    Let $\Sigma^n$ be an area-minimizing hypersurface in a manifold $M^{n+1}$ with singularities. Suppose the tangent cone of $M^{n+1}$ at a given point $p$ is a metric cone $C_p$ (i.e., $C_p$ is scaling invariant), and $\eta_{p,\lambda_i,\#}M$ converges to its tangent cone $C_p$ in the $C^1$ sense locally away from $\text{Sing}(C_p)$. Then the blow-up limit $\Sigma_{\infty}$ of $\Sigma$ is minimizing in $C_p$. Moreover, if we continue to blow up $\Sigma_{\infty}$ in $C_p$ at $0$, then we will obtain a minimizing cone $\Sigma_{\infty}'$ in $C_p$.
\end{corollary}
\begin{proof}
    The proof is standard using the monotonicity formula (Lemma \ref{mono}) and the compactness lemma $\ref{cpt}$. (See \cite{Simon1983} or \cite{Giusti1984}.)\par
    By assumption, the (first) blow-up limit of $\Sigma$ is $\Sigma_{\infty}$. That is, there exists a subsequence $\{\lambda_i\}$ (not labeled) such that $\Sigma_i:=\eta_{p,\lambda_i,\#}\Sigma\rightharpoonup \Sigma_{\infty}$. By the monotonicity formula for $\Sigma_{\infty}$ in $C_p$, 
    $$\Theta_{\Sigma_{\infty}}(0,r)=\frac{Vol(B_r\cap \Sigma_{\infty})}{\omega_nr^n}\quad\text{ is increasing in }r,$$
    and therefore,
    $$\Theta_{\Sigma_{\infty}}(0)=\lim_{r\rightarrow0}\Theta_{\Sigma}(0,r)\quad\text{ exists.}$$
    
    Now since $\Sigma_{\infty}'$ is the blow-up limit of $\Sigma_{\infty}$ in $C_p$ at $0$, by the weak convergence, for any $0<s<t$:
    $$\frac{Vol(B_s(p)\cap \Sigma_{\infty})}{\omega_ns^n}=\frac{Vol(B_t(p)\cap \Sigma_{\infty})}{\omega_nt^n}=\Theta_{\Sigma_{\infty}}(0).$$
    By the monotonicity formula, this implies that $\lvert x^{\bot}\rvert=0$ for any $x\in \Sigma_{\infty}'$, which means that $\Sigma_{\infty}'$ is a cone. And the minimizing property of $\Sigma_{\infty}$ and  $\Sigma_{\infty}'$ comes from Lemma \ref{cpt}. (Note that in the second blow-up step, the ambient space $C_p$ is invariant under scaling.)
\end{proof}

Next, we show that the construction of tangent cylinders and the classical dimension reduction process are still applicable, provided the tangent cones of the ambient space have good properties.

\begin{lemma}\label{cylinder}
    Let $C$ be a metric cone, i.e., $\eta_{0,\lambda,\#}C=C$ for any $\lambda>0$. And suppose that the tangent cone $C'$ of $C$ at a point $0\neq p\in C$ is always a metric cone. (The tangent cones may not be unique, but they are all cones.) Then $C'$ is a tangent cylinder of $C$ at $p$. That is, $\eta_{kp,1,\#}C'=C'$, for all $k\in\mathbb{R}$. 
\end{lemma}
\begin{proof}
    We have by the assumptions that for some $\lambda_i\rightarrow0$,
    $$\eta_{p,\lambda_i,\#}C=\frac{C-p}{\lambda_i}\rightharpoonup C'.$$
    Since $C$ is a cone,
    $$C'-p=\lim_{i\rightarrow\infty}\frac{C-p}{\lambda_i}-p=\lim_{i\rightarrow\infty}\frac{C-(\lambda_i+1)p}{\lambda_i}=\lim_{j\rightarrow\infty}\frac{(\lambda_{i_j}+1)(C-p)}{\lambda_{i_j}}=C'',$$
    for some other tangent cone $C''$. Note that by assumption, $C'$ and $C''$ are both cones. Therefore, for any $\lambda>0$,
    $$\eta_{\lambda^{-1}p,1,\#}C'=\eta_{p,\lambda,\#}C'=\eta_{0,\lambda,\#}C''=C''=\eta_{p,1,\#}C'.$$
    This implies that for any $k\in\mathbb{R}$, $\eta_{kp,1,\#}C'=C'$. Thus, we obtain $C'$ is a tangent cylinder, as required.
\end{proof}

We remark here that using this lemma, the splitting property in (5) of the definition of the family $\mathcal{F}$ can be replaced by the assumption that the tangent cones of any tangent cone are all metric cones.

The following proposition for minimizing cylinders also extend to singular ambient space cases.

\begin{prop}\label{split}
    $\Sigma^n\times\mathbb{R}$ is area-minimizing in $M^{n+1}\times\mathbb{R}$ if and only if $\Sigma^n$ is area-minimizing in $M^{n+1}$.
\end{prop}
\begin{proof}
    On the one hand, suppose $\Sigma^n$ is area-minimizing in $M^{n+1}$. Then for any $N$ in $M^{n+1}\times\mathbb{R}$ which has the same boundary as $\Sigma^n\times\mathbb{R}$ and differs from $\Sigma^n\times\mathbb{R}$ only on a compact set $K=K'\times [a,b]$, we have by the coarea formula:
    \begin{align*}
        &~\text{Area}(N\cap K)-\text{Area}((\Sigma\times\mathbb{R})\cap K)\\
        =&\int_a^b\left(\text{Area}(N\cap K')\frac{1}{\lvert\nabla^N h\rvert}-\text{Area}(\Sigma\cap K)\right)dh\\
        \geq& \int_a^b\left(\text{Area}(N\cap K')-\text{Area}(\Sigma\cap K')\right)dh\\
        \geq&~0,
    \end{align*}
    where $h$ is the height function and we have used that $\text{Area}(N\cap K')\geq\text{Area}(\Sigma\cap K')$, since $\Sigma$ is minimizing. And this completes the proof of the \textit{if} part of the proposition.
    
    On the other hand, suppose $\Sigma^n\times\mathbb{R}$ is area-minimizing in $M^{n+1}\times\mathbb{R}$. If $\Sigma$ is not area-minimizing in $M$, then there exists $N$ in $M$ with the same boundary of $\Sigma$ such that, $N$ differs from $\Sigma$ only on a compact set $K$, and $\text{Area}(N\cap K)<\text{Area}(\Sigma\cap K)$. Now it is possible to construct an $\tilde{N}$ such that:
    \begin{equation*}
    \tilde{N}=\left\{
    \begin{aligned}
        &N,\quad &\lvert h\rvert\le R,\\
        &I(N,\Sigma),\quad &R<\lvert h\rvert<R+\epsilon,\\
        &\Sigma,\quad &\lvert h\rvert>R+\epsilon.
    \end{aligned}
    \right
    .
    \end{equation*}
    where $I(N,\Sigma)$ is some hypersurface connecting $\Sigma\times\{R+\epsilon\}$ and $N\times\{R\}$. Fix $\epsilon$ small and let $R$ be large enough, we will find that $\tilde{N}$ has less area than $\Sigma\times\mathbb{R}$. However, $\tilde{N}$ agrees with $\Sigma\times\mathbb{R}$ on the boundary and outside the compact set $K\times[R-\epsilon,R+\epsilon]$. A contradiction. This finishes the proof of the \textit{only if} part.
\end{proof}

Now we are ready to prove our main regularity theorem \ref{main1}. The proof is based on the classical blow-up argument and the basic case discussed in Section 3.

\begin{proof}[Proof of Theorem \ref{main1}]
    Assume (6a) holds in the definition of $M\in\mathcal{F}$. First observe that, by the classical regularity results for area-minimizing hypersurfaces, we have: 
    $$\text{dim}\left(\text{Sing}(\Sigma)\cap (M\setminus \text{Sing}(M))\right)\le n-7\le n-3,$$
    since $M\setminus \text{Sing}(M)$ is a smooth manifold. Therefore, we only need to control $\text{dim}(\text{Sing}(\Sigma)\cap\text{Sing}(M)).$ 

    Suppose there exists $d>0$ such that 
    $$\mathcal{H}^{n-3+d}(\text{Sing}(\Sigma)\cap\text{Sing}(M))>0.$$
    
    By Proposition \ref{hkae}, we can choose $p\in \text{Sing}(\Sigma)\cap \text{Sing}(M)$ and a sequence $\{r_j\}\rightarrow0$ such that:
    $$\lim_{j\rightarrow\infty}\frac{\mathcal{H}_{n-3+d}^{\infty}\left(\text{Sing}(\Sigma)\cap\text{Sing}(M)\cap B_{r_j}(p)\right)}{\omega_{n-3+d}r^{n-3+d}}\geq2^{-(n-3+d)}.$$
    Then blow up $\Sigma$ in $M$ at $p$ as in Corollary \ref{exist} to obtain a minimizing hypersurface $\Sigma_1$ in the tangent cone $C$ of $M$ at $p$ with
    $$\mathcal{H}^{n-3+d}(\text{Sing}(\Sigma_1)\cap\text{Sing}(C))>0.$$
    Again by Proposition \ref{hkae} and Corollary \ref{exist}, we can blow up $\Sigma_1$ at some point $p_1\in\text{Sing}(\Sigma_1)\cap\text{Sing}(C)$ $(p_1\neq p)$ and obtain a minimizing $\Sigma_2$ in the tangent cone $C_1$ of $C$ at $p_1$ with
    $$\mathcal{H}^{n-3+d}(\text{Sing}(\Sigma_2)\cap\text{Sing}(C_1))>0,$$
    and so on. Note that compared to $C$, $C_1$ is not only a cone but also a cylinder (see Lemma \ref{cylinder}), which splits out a line passing through $p_1$. Therefore, after blowing up $\Sigma$ finitely many times, we will finally get a minimizng hypersurface $\Sigma_m$ in a tangent cylinder $C_{m-1}$, where $C_{m-1}=\RR^{m-1}\times C_0$ for some cone $C_0$. Moreover, we have:
    $$\mathcal{H}^{n-3+d}(\text{Sing}(\Sigma_m)\cap\text{Sing}(C_{m-1}))>0,$$
    and each singular point belongs to the spine of $C_{m-1}$. (A point $q$ belongs to the spine of a cylinder $C$, if $\eta_{q,1,\#}C=C$.) 

    Now that the ambient space $C_{m-1}$ is sufficiently regular, we can further blow up $\Sigma_m$ at some point $p_m$ (note that $\Sigma_m$ need not be a cone) to obtain a minimizing cone $\Sigma_{m+1}$ in $C_{m-1}$, by Corollary \ref{cor2}.
    Since $\Sigma_{m+1}$ and $C_{m-1}$ are both cones, $\text{Sing}(\Sigma_{m+1})\cap\text{Sing}(C_{m-1})$ splits out a line $l$ passing through $p_m$, provided $n-3+d>0$.
    
    Next, choose $p_{m+1}\in l$, $p_{m+1}\neq p_m$, and blow up $\Sigma_{m+1}$ in $C_{m-1}$ at $p_{m+1}$. Then by Lemma \ref{cylinder} and Corollary \ref{cor2}, we obtain a minimizing cylinder $\Sigma_{m+2}$ in the tangent cylinder $C_{m-1}$. Here we use that $C_{m-1}$ is a cone and that the monotonicity formula for minimal surfaces holds in this case.
    
    If we write $\Sigma_{m+2}=\Sigma_{m+3}\times\mathbb{R}$ and $C_{m-1}=C_m\times\mathbb{R}$, then by Proposition \ref{split}, we obtain a minimizing cone $\Sigma_{m+3}$ in $C_m$ with
    $$\mathcal{H}^{n-4+d}(\text{Sing}(\Sigma_{m+3})\cap\text{Sing}(C_m))>0.$$
    Recall our definition of the family $\mathcal{F}$, and we see that $C_m$ inherits all the properties from $C_{m-1}$. Therefore, arguing inductively as above, we will finally obtain a singular minimizing $2$-dimensional $\Sigma_0$ in a $3$-dimensional oriented scalar nonnegative cone $C_0$ with an isolated singular vertex $p$, a contradiction to Corollary \ref{cor}. Consequently,
    $\text{dim}(\text{Sing}(\Sigma^n))\le n-3.$\par
    On the other hand, if (6b) holds in the definition of $M\in\mathcal{F}$ and suppose $\text{dim}(\text{Sing}(\Sigma^n))>n-2$, then a similar dimension reduction argument will lead to an existence of a minimizing geodesic on a 2-dimensional cone $C_0$ passing through the vertex $p$ of $C_0$, where $\Theta_{C_0}(p)\le 1-\epsilon<1$. Since any $2$-dimensional cone is developable, it is not hard to see that this is impossible. Hence we have $\text{dim}(\text{Sing}(\Sigma^n))\le n-2$ in this case.    
\end{proof}

\section{An Example For Higher Dimensions}
In this section we compute a concrete example (Example \ref{main2}) to show that the dimension bound in Theorem \ref{main1} cannot be improved. This example shows the existence of 3-dimensional area-minimizing hypercones in 4-d scalar positive ambient spaces with singularities. \par

As is pointed out in the introduction, this example has already been discussed in \cite{Morgan2002}. Here, we provide an alternative approach which is similar to the discussion in Section 3.

Let $C(S^n(\lambda))$ be the standard cone with vertex $O$ and cross section $S^n(\lambda)$, where $0<\lambda\le 1$ is the radius of the sphere $S^n$. And let $S^{n-1}(\lambda)$ be the equator of $S^n(\lambda)$. Consider the following Plateau problem:
\begin{p}\label{q}
    What does the surface of least area look like among all the surfaces in the ambient space $C(S^n(\lambda))$ with boundary $S^{n-1}(\lambda)$? 
\end{p}
We solve part of this problem first. Denote the solution surface to this problem by $\Sigma^n$. And we will focus on the case $2\le n\le6$ and $0<\lambda<1$.

We start by discussing whether $\Sigma$ intersects the singular point $O$.

\begin{prop}\label{prop1}
    If $\Sigma$ is not smooth, then $O\in \Sigma$ and the totally geodesic hypercone $C(S^{n-1}(\lambda))$ is a solution to Problem \ref{q}.
\end{prop}
\begin{proof}
    When $2\le n\le6$, classical regularity theory shows that $\Sigma \cap C(S^n(\lambda)) \setminus \{O\}$ is smooth. Therefore, if $\Sigma$ is not smooth, $\Sigma$ must contain $O$.\par
    To prove the second statement, we blow up $\Sigma$ in $C(S^n(\lambda))$ and obtain an area-minimizing cone $\Sigma_1$ in $C(S^n(\lambda))$ by Corollary \ref{cor2}. Now we claim that any minimal cone different from $C(S^{n-1}(\lambda))$ in $C(S^n(\lambda))$ is unstable. And this will force $\Sigma_1$ to be identical to $C(S^{n-1}(\lambda))$. Thus, $C(S^{n-1}(\lambda))$ is area-minimizing and can be chosen to solve the initial Plateau problem.\\
    Proof of the claim: The proof is the same as the classical proof for minimal cones using Simons' inequality. See Appendix for details or refer to \cite{Simons1968},\cite{Simon1983},\cite{Yau1975}.
\end{proof}

Therefore, we only need to consider $\Sigma\subseteq C(S^n(\lambda))\setminus \{O\}$ being smooth.

\begin{lemma}\label{rot}
    If $\Sigma$ is smooth, then $\Sigma$ is rotationally symmetric in the sense that $\Sigma$ is invariant under any rotation of $C(S^n)$ which fixes the boundary $S^{n-1}(\lambda)$. 
\end{lemma}
We will use an analogue of the moving plane method.

Let $\mathbb{R}^{n+1}$ be given the metric:
$$g_{\lambda}=dr^2+\lambda^2r^2g_{S^n},$$
where $(r,x)$ is the polar coordinate system. Thus, $(\mathbb{R}^{n+1},g_{\lambda})$ is isometric to the cone $C(S^n(\lambda))$. And then the given boundary $S^{n-1}(\lambda)$ in Problem \ref{q} will be $\partial B_1\cap \mathbb{R}^n \times \{0\}$. And we will use this model from now on. 

Recall that $\Sigma$ is the solution to Problem \ref{q} and $\partial\Sigma=\partial B_1\cap \mathbb{R}^n \times \{0\}$.

\begin{prop}
    $\Sigma$ lies inside the ball $B_1\subseteq\mathbb{R}^{n+1}$.
\end{prop}
\begin{proof}
    We show that the square of the distance function $r^2(x)=\lvert x\rvert^2$ on $\Sigma$ is subharmonic. And the proposition follows as a consequence. Indeed,
    $$\Delta_{\Sigma}(r^2)=\text{div}_{\Sigma}(2r\nabla_{\Sigma}r)=2\text{div}_{\Sigma}(r\nabla_{C(S^n(\lambda))} r)=2\text{div}_{\Sigma}(X)=2n,$$
    where $X$ is the position vector field. Here in the second equality we use $\text{div}_{\Sigma}(N)=0$, for all normal vector $N$, since $\Sigma$ is minimal. And the reason for the last equality is the same as that in the monotonicity formula. See Lemma \ref{mono}. (To check these equalities, it is actually better to embed $C(S^n(\lambda))$ into the Euclidean space $\mathbb{R}^{n+2}$.)
\end{proof}

\begin{proof}[Proof of Lemma \ref{rot}]
    Fix any unit vector $\gamma\in\mathbb{R}^n$, and consider a family of rotating hyperplanes $$\mathcal{P_{\gamma}}=\{P_{\gamma}(t)\mid 0\le t\le\pi\},$$ 
    where 
    $$P_{\gamma}(t)=\{(x,y)\in\mathbb{R}^n\times\mathbb{R}=\mathbb{R}^{n+1}\mid (x,y)\bot (\gamma\sin t,\cos t)\}.$$
    For each hyperplane $P_{\gamma}(t)$, let
    $$\Sigma_t^+=\{(x,y)\in\Sigma\mid (x,y)\cdot (\gamma\sin t,\cos t)>0\},$$
    $$\Sigma_t^-=\{(x,y)\in\Sigma\mid (x,y)\cdot (\gamma\sin t,\cos t)>0\}.$$
    Then reflect the surface $\Sigma_t^-$ across the plane $P_{\gamma}(t)$ to get the reflection $\Sigma_t^{-*}$. 
    
    We start at $t=0$. If $\Sigma_0^-\neq\emptyset$, we define a new surface $\Sigma '=\overline{\Sigma_0^{+}\cup \Sigma_0^{-*}}$. Then $\Sigma'$ will also be a solution to Problem \ref{q}, since Area($\Sigma'$)=Area($\Sigma$) and $\Sigma$ is area-minimizing. Moreover, since $\Sigma$ is away from $O$, $\Sigma'$ must be smooth according to the regularity of area-minimizing hypersurfaces. Consequently, we can replace $\Sigma$ by $\Sigma'$ if necessary and assume $\Sigma_0^-=\emptyset$ without loss of generality. Similarly, $\Sigma_{\pi}^-=\Sigma$. 
    
    Now as $t$ increases from $0$ to $\frac{\pi}{2}$, we will find a $t_0$ which is the first time that $\Sigma_t^{-*}$ touches $\Sigma_t^+$. At this time, since the ambient metric is invariant under the reflection, $\Sigma_{t_0}^{-*}$ and $\Sigma_{t_0}^+$ satisfy the same minimal surface equation. In addition, $\Sigma_{t_0}^{-*}$ lies on one side of $\Sigma_{t_0}^+$ and touches $\Sigma_{t_0}^+$ at some point. Therefore, by the strong maximum principle, we conclude that the touching point cannot be an interior point unless $\Sigma_{t_0}^{-*}$ is identical to $\Sigma_{t_0}^+$. Hence, the only possiblility is that $t_0=\frac{\pi}{2}$ and $\Sigma_{\pi/2}^{-*}$ is identical to $\Sigma_{\pi/2}^+$.

    Till now we have showed that $\Sigma$ is invariant under the reflection across the plane $P_{\gamma}(\frac{\pi}{2})$. Since $\gamma$ is chosen arbitrarily, we conclude that $\Sigma$ is rotationally symmetric.
\end{proof}

According to this lemma, we only need to find the surface of least area among rotationally symmetric ones. And this reduce the initial problem to a concrete ODE which we now establish.

We come back to our model $(\mathbb{R}^{n+1},g_{\lambda})$. Since the area-minimizing solution $\Sigma$ is rotationally symmetric, it can be represented by an one-variable function $r=f(\theta)$ using the polar coordinate system. Here $0\le\theta\le\frac{\pi}{2}$ is the polar angle, and $f$ is the distance from $O$.

Then the area functional for $\Sigma$ is:
$$\text{Area}(\Sigma)=n\omega_n\lambda^{n-1}\int_0^{\frac{\pi}{2}}\sqrt{(f')^2+\lambda^2f^2}f^{n-1}\cos^{n-1}\theta d\theta.$$
And the first variation gives a second order ODE for $f=f(\theta)$:
\begin{equation}\label{eq1}
    \left(\frac{f'\cos^{n-1}{\theta}}{\sqrt{(f')^2+\lambda^2f^2}}\right)'-\frac{n\lambda^2f\cos^{n-1}{\theta}}{\sqrt{(f')^2+\lambda^2f^2}}=0.
\end{equation}
The boundary condition for $\Sigma$ becomes the initial condition for $f$: 
\begin{equation}\label{eq2}
    f(0)=1,\quad f'(0)=A.
\end{equation}

Recall that by Proposition \ref{prop1}, if $\Sigma$ is not smooth, then $\Sigma=C(S^{n-1}(\lambda))$, and
$$\text{Area}(C(S^{n-1}(\lambda)))=\omega_n\lambda^{n-1}.$$
Therefore, whether the solution $\Sigma$ to Problem \ref{q} is smooth or not depends on whether there exists a solution  to (\ref{eq1}) and (\ref{eq2}) such that:
\begin{equation}\label{eq3}
    S_{n,\lambda}(f) = \int_0^{\frac{\pi}{2}}\sqrt{(f')^2+\lambda^2f^2}f^{n-1}\cos^{n-1}\theta d\theta<\frac{1}{n}.
\end{equation}

\begin{lemma}\label{lem2}
    If the solution $f$ to the equation (\ref{eq1}) with some initial data $A$ in (\ref{eq2}) can be extended to $\frac{\pi}{2}$, and is a minimizer for the functional $S_{n,\lambda}$ in (\ref{eq3}), then the inequality (\ref{eq3}) holds. And therefore, $C(S^{n-1}(\lambda))$ is not area-minimizing in $C(S^n(\lambda))$.
\end{lemma}
\begin{proof}
    First observe that, if $f$ is a minimizer, then $f$ must be a decreasing function. Thus, we only need to consider the initial data $f'(0)=A\le 0$.

    Suppose the solution $f$ to (\ref{eq1}) and (\ref{eq2}) is defined on the whole interval $[0,\frac{\pi}{2}]$. Then
    \begin{align*}
        S_{n,\lambda}(f)&=\int_0^{\frac{\pi}{2}}\sqrt{(f')^2+\lambda^2f^2}f^{n-1}\cos^{n-1}\theta d\theta\\
        &=\int_0^{\frac{\pi}{2}}\left((f')^2+\lambda^2f^2\right)f^{n-2}\left(\frac{f\cos^{n-1}\theta}{\sqrt{(f')^2+\lambda^2f^2}} \right) d\theta\\
        &=\int_0^{\frac{\pi}{2}}\frac{(f')^2f^{n-1}\cos^{n-1}\theta}{\sqrt{(f')^2+\lambda^2f^2}}+\lambda^2f^n\left(\frac{f\cos^{n-1}\theta}{\sqrt{(f')^2+\lambda^2f^2}} \right) d\theta\\
        &=\frac{1}{n}\int_0^{\frac{\pi}{2}}\frac{(nf'f^{n-1})f'\cos^{n-1}\theta}{\sqrt{(f')^2+\lambda^2f^2}}+f^n\left(\frac{f'\cos^{n-1}{\theta}}{\sqrt{(f')^2+\lambda^2f^2}}\right)'d\theta\\
        &=\frac{1}{n}\left(\frac{f'f^n\cos^{n-1}{\theta}}{\sqrt{(f')^2+\lambda^2f^2}}\right)\Bigg|_{0}^{\frac{\pi}{2}}\\
        &=\frac{-A}{n\sqrt{A^2+\lambda^2}}<\frac{1}{n},
    \end{align*}
    where in the forth equality we use the equation (\ref{eq1}).
\end{proof}

\begin{lemma}\label{lem3}
    If the solution $f$ to the equation (\ref{eq1}) with any initial data $A$ chosen in (\ref{eq2}) cannot be extended to $\frac{\pi}{2}$, then $S_{n,\lambda}\geq\frac{1}{n}$ holds for all $f$ with $f(0)=1$. And consequently, $C(S^{n-1}(\lambda))$ is area-minimizing in $C(S^n(\lambda))$.
\end{lemma}
\begin{proof}
    Suppose $C(S^{n-1}(\lambda))$ is not area-minimizing in $C(S^n(\lambda))$, then by Proposition \ref{prop1}, the solution $\Sigma$ to Problem \ref{q} is smooth (away from $O$). Furthermore, $\Sigma$ is rotationally symmetric by Lemma \ref{rot}. Hence, the function $f$ representing $\Sigma$ is a minimizer for the functional $S_{n,\lambda}$. And this $f$ is a solution of (\ref{eq1}) defined on the whole interval $[0,\frac{\pi}{2}]$, a contradiction to the assumption. 
\end{proof}
\begin{remark}\label{rem}
    There is only one thing in the proof that we should worry about. That is, $\Sigma$ may be a surface such that $f'=-\infty$ somewhere. This problem is handled under a change of variables in the following part.
\end{remark}

In order to verify Example \ref{main2}, we need further study on the behavior of the ODE (\ref{eq1}).

Let us make the following change of variables:
\begin{equation}\label{eq4}
    \left(\frac{f'}{\sqrt{(f')^2+\lambda^2f^2}}\right)=-\cos H, \quad
    \left(\frac{\lambda f}{\sqrt{(f')^2+\lambda^2f^2}}\right)=\sin H.
\end{equation}
Then the initial second order equation (\ref{eq1}) and (\ref{eq2}) for $f$ becomes a first order equation for $H$:
\begin{equation}\label{eq5}
    H'=n\lambda-(n-1)\tan\theta\cot H,\quad H(0)=\arccos\left(\frac{-A}{\sqrt{A^2+\lambda^2}}\right).
\end{equation}
And by (\ref{eq4}), we have the relation between $f$ and $H$:
\begin{equation}\label{eq6}
    \frac{f'}{f}=-\lambda\cot H.
\end{equation}
We are most interested in the domain:
$$D=\left\{(\theta,H)\in\mathbb{R}^2\mid0<\theta<\frac{\pi}{2},\:0<H<\frac{\pi}{2}\right\},$$
since $f'\le0$ implies $H\le\frac{\pi}{2}$.

First note that, if $\theta\rightarrow0$, $H(\theta)\rightarrow0$, then by Taylor expansion we have:
$$\tan{\theta}\cot{H(\theta)}\sim \frac{1}{H'(0)},$$
and therefore we obtain:
$$H'(0)=n\lambda-\frac{n-1}{H'(0)}.$$
This gives the condition on $n$ and $\lambda$ which appears in Theorem \ref{main2}:
\begin{equation}\label{eq7}
    n\lambda\geq 2\sqrt{n-1}.
\end{equation}
And the analysis above motivates us to consider such a bound in Example \ref{main2}.

Next, assume (\ref{eq7}) holds (hence $n\geq 3$) and consider the line
$$l=\{(\theta,H)\mid H=c\theta\},$$
where $c=\frac{n\lambda+\sqrt{n^2\lambda^2-4(n-1)}}{2}>1$ is the larger root of the equation 
$$c=n\lambda-\frac{n-1}{c}.$$
Using the basic inequality $(c>1)$:
$$\tan c\theta>c\tan\theta,\quad \theta\in\left(0,\frac{\pi}{2}\right),$$
we obtain by the equation (\ref{eq5}) that, on the line $l$ in $D$,
$$H'=n\lambda-(n-1)\tan\theta\cot c\theta>n\lambda-\frac{n-1}{c}=c.$$
This shows that any solution $H$ to (\ref{eq5}) cannot cross $l$ from left to right in the region $D$. As a consequence, any solution $H$ to (\ref{eq5}) with initial data $0<H(0)\le\frac{\pi}{2}$ cannot be extended to $\frac{\pi}{2}$ in $D$. Therefore, we can apply Lemma \ref{lem3} to conclude that $C(S^{n-1}(\lambda))$ is area-minimizing in $C(S^n(\lambda))$, as long as $n\lambda\geq 2\sqrt{n-1}$. 

However, as mentioned in Remark \ref{rem}, we should consider the case $f'=-\infty$, which is equivalent to $H=0$. It may happen that the $H$ corresponding to the minimizing $\Sigma$ turns to $0$ somewhere. If $H(\theta_0)=0$ and $\theta_0>0$, then by (\ref{eq5}) $H'(\theta_0)=-\infty$. Therefore, $\{0<\theta_0<\frac{\pi}{2}\mid H(\theta_0)=0\}$ are isolated points. Since $H$ is induced by the smooth surface $\Sigma$, $H$ is defined on $[0,\frac{\pi}{2}]$. And then we must have $H(0)=0$ according to the argument above. In this case, for any $g$ with $g(0)=1$,
\begin{align*}
    S_{n,\lambda}(g)=&\int_0^{\frac{\pi}{2}}\sqrt{(g')^2+\lambda^2g^2}g^{n-1}\cos^{n-1}\theta d\theta\\
    \geq&\int_0^{\frac{\pi}{2}}(-g'\cos H+\lambda g\sin H)g^{n-1}\cos^{n-1}\theta d\theta\\
    =&-\frac{1}{n}g^n\cos^{n-1}\theta\cos H\bigg|_{0}^{\frac{\pi}{2}}+\frac{1}{n}\int_0^{\frac{\pi}{2}}g^nd(\cos^{n-1}\theta\cos H)\\
    &+\int_0^{\frac{\pi}{2}}\lambda g^n\cos^{n-1}\theta\sin Hd\theta\\
    =&~\frac{1}{n}+\frac{1}{n}\int_0^{\frac{\pi}{2}}g^n\cos^{n-1}\theta\sin H\left(n\lambda-(n-1)\tan\theta\cot H-H' \right)d\theta\\
    =&~\frac{1}{n}.
\end{align*}
Again, we get the conclusion of Lemma \ref{lem3}. And these arguments together complete the proof of the \textit{if} part of Example \ref{main2}.

Finally, we assume that
\begin{equation}\label{eq8}
    n\lambda< 2\sqrt{n-1},
\end{equation}
and we will show that there exists some piecewise smooth $f$ satisfying (\ref{eq2}), and $S_{n,\lambda}(f)<\frac{1}{n}$.

Define \begin{equation*}
    f(\theta)=\left\{
    \begin{aligned}
        &e^{-\mu\theta},\quad &0\le\theta\le\delta,\\
        &\alpha e^{-\lambda g(\theta)},\quad &\delta<\theta\le\frac{\pi}{2}.
    \end{aligned}
    \right
    .
\end{equation*}
where
\begin{equation*}
    \mu=-\frac{\ln\alpha}{\delta},\quad g(\theta)=\int_{\delta}^{\theta}\frac{dt}{\sqrt{\frac{1}{\cos^{2n-2}t}-1}}.
\end{equation*}
And $\delta$, $\alpha\in (0,1)$ are constants to be determined. Note that $f$ is continuous and piecewise smooth, $f(0)=1$, $f(\delta)=\alpha$.
By definition, when $\delta<\theta\le\frac{\pi}{2}$,
$$f'(\theta)=-\lambda g'(\theta)f(\theta)=\frac{-\lambda f(\theta)}{\sqrt{\frac{1}{\cos^{2n-2}\theta}-1}}.$$
That is, 
$$(f')^2+\lambda^2f^2=\frac{(f')^2}{\cos^{2n-2}\theta}.$$
Recall that we are only interested in the case $f'\le0$. So we obtain from (\ref{eq3}) and the difinition of $f$ that
\begin{align*}
    S_{n,\lambda}(f)&=\int_0^{\delta}\sqrt{(f')^2+\lambda^2f^2}f^{n-1}\cos^{n-1}\theta d\theta+\int_{\delta}^{\frac{\pi}{2}}\left(\frac{-f'}{\cos^{n-1}\theta}\right)f^{n-1}\cos^{n-1}\theta d\theta\\
    &=\int_0^{\delta}\sqrt{\mu^2+\lambda^2}e^{-n\mu\theta}\cos^{n-1}\theta d\theta-\int_{\delta}^{\frac{\pi}{2}}f'f^{n-1}d\theta\\
    &\le\sqrt{\mu^2+\lambda^2}\int_0^{\delta}e^{-n\mu\theta}d\theta+\frac{1}{n}\left(f(\delta)^n-f\left(\frac{\pi}{2}\right)^n\right)\\
    &=\frac{\sqrt{\mu^2+\lambda^2}}{n\mu}\left(1-e^{-n\mu\delta}\right)+\frac{1}{n}\alpha^n-\frac{1}{n}f\left(\frac{\pi}{2}\right)^n\\
    &=\frac{1}{n}(1-\alpha^n)\sqrt{1+\frac{\lambda^2\delta^2}{\ln^2\alpha}}+\frac{1}{n}\alpha^n-\frac{1}{n}\alpha^ne^{-n\lambda g(\pi/2)}\\
\end{align*}
Since
\begin{align*}
    g\left(\frac{\pi}{2}\right)&=\int_{\delta}^{\frac{\pi}{2}}\frac{dt}{\sqrt{\frac{1}{\cos^{2n-2}t}-1}}\\
    &=\int_{\delta}^{\frac{\pi}{2}}\frac{dt}{\sqrt{\frac{1}{\cos^2t}-1} \sqrt{\frac{1}{\cos^{2n-4}t}+\dots+1}}\\
    &\le \int_{\delta}^{\frac{\pi}{2}}\frac{\cos t\:dt}{\sqrt{n-1}\sin t}\\
    &=\frac{-\ln\sin\delta}{\sqrt{n-1}},
\end{align*}
we obtain
$$S_{n,\lambda}(f)\le\frac{1}{n}(1-\alpha^n)\sqrt{1+\frac{\lambda^2\delta^2}{\ln^2\alpha}}+\frac{1}{n}\alpha^n-\frac{1}{n}\alpha^n(\sin\delta)^{\frac{n\lambda}{\sqrt{n-1}}}.$$
Now fix $\alpha<1$, then
\begin{align*}
    S_{n,\lambda}(f)&\le\frac{1}{n}(1-\alpha^n)\left(1+\frac{\lambda^2\delta^2}{2\ln^2\alpha}+o(\delta^2)\right)+\frac{1}{n}\alpha^n-\frac{1}{n}\alpha^n(\sin\delta)^{\frac{n\lambda}{\sqrt{n-1}}}\\
    &=\frac{1}{n}+\frac{\lambda^2(1-\alpha^n)}{2n\ln^2\alpha}\delta^2-\frac{1}{n}\alpha^n(\sin\delta)^{\frac{n\lambda}{\sqrt{n-1}}}+o(\delta^2).
\end{align*}
We see that if $\frac{n\lambda}{\sqrt{n-1}}<2$, then we can choose sufficiently small $\delta$ such that $S_{n,\lambda}(f)<\frac{1}{n}$. Thus, there exists a rotationally symmetric piecewise smooth surface $\Sigma$ corresponding to $f$ such that $\partial\Sigma=\partial C(S^{n-1}(\lambda))$ but $\text{Area}(\Sigma)<\text{Area}(C(S^{n-1}(\lambda)))$. Consequently, $C(S^{n-1}(\lambda))$ is not area-minimizing in $C(S^n(\lambda))$. This completes the proof of the \textit{only if} part of Example \ref{main2}.

\section*{Appendix}
\setcounter{equation}{0}
\renewcommand\theequation{\arabic{equation}}

In the first part, we compute the curvatures of a cone and its cross section.

Let $(C(M^n),g)$ be a cone with vertex $O$ and cross section $M^n$, where
$$g=dt^2+t^2g_M.$$
And let $x^i$ be coordinates on $M$ such that
$$g_M=h_{ij}dx^idx^j.$$
Denote
$$e_i=\frac{\partial}{\partial x^i},\quad e_t=\frac{\partial}{\partial t},\quad \nabla=\nabla^{M},\quad \overline{\nabla}=\nabla^{C(M)}.$$
Then by direct computation, the only nonzero Christoffel symbols are:
\begin{align*}
    &\overline{\Gamma}_{ij}^k=\frac{1}{2}g^{kl}(\partial_ig_{jl}+\partial_jg_{il}-\partial_lg_{ij})=\Gamma_{ij}^k,\\
    &\overline{\Gamma}_{it}^j=\frac{1}{2}g^{jl}(\partial_ig_{tl}+\partial_tg_{il}-\partial_lg_{it})=\frac{1}{2t^2}h^{jl}(2t\,h_{il})=\frac{1}{t}\delta_{ij},\\
    &\overline{\Gamma}_{ij}^t=-\frac{1}{2}\partial_tg_{ij}=-t\,h_{ij}.
\end{align*}
Therefore we have the rules for computing the connection:
\begin{align*}
    &\overline{\nabla}_{e_i}e_j=\Gamma_{ij}^ke_k-t\,h_{ij}e_t,\\ &\overline{\nabla}_{e_i}e_t=\overline{\nabla}_{e_t}e_i=\frac{1}{t}e_i,\\
    &\overline{\nabla}_{e_t}e_t=0.
\end{align*}
And the curvatures are:
\begin{align*}
    \overline{R}_{ijji}&=\langle \overline{\nabla}_{e_i}\overline{\nabla}_{e_j}e_j,e_i \rangle-\langle \overline{\nabla}_{e_j}\overline{\nabla}_{e_i}e_j,e_i \rangle\\
    &=\langle \overline{\nabla}_{e_i}(\Gamma_{jj}^ke_k-t\,h_{jj}e_t),e_i \rangle-\langle \overline{\nabla}_{e_j}(\Gamma_{ij}^ke_k-t\,h_{ij}e_t),e_i \rangle\\
    &=(\partial_i\Gamma_{jj}^k)g_{ik}+\Gamma_{jj}^k\Gamma_{ik}^lg_{il}-t\,h_{jj}\frac{1}{t}g_{ii}-(\partial_j\Gamma_{ij}^k)g_{jk}-\Gamma_{ij}^k\Gamma_{jk}^lg_{il}+t\,h_{ij}\frac{1}{t}g_{ij}\\
    &=t^2R_{ijji}-t^2(h_{ii}h_{jj}-h_{ij}^2),
\end{align*}
\begin{align*}
    \overline{R}_{itti}=\langle \overline{\nabla}_{e_i}\overline{\nabla}_{e_t}e_t,e_i \rangle-\langle \overline{\nabla}_{e_t}\overline{\nabla}_{e_i}e_t,e_i \rangle=-\langle \overline{\nabla}_{e_t}(\frac{1}{t}e_i),e_i \rangle=\frac{1}{t^2}g_{ii}-\frac{1}{t^2}g_{ii}=0.
\end{align*}
\begin{equation}\label{sec}
    \overline{K}(e_i,e_j)=\frac{K(e_i,e_j)-1}{t^2},\quad \overline{K}(e_i,e_t)=0.
\end{equation}
\begin{equation}\label{Ric}
    \overline{\text{Ric}}(e_i,e_i)=\frac{\text{Ric}(e_i,e_i)-(n-1)}{t^2},\quad \overline{\text{Ric}}(e_t,e_t)=0.
\end{equation}

Next, we find out the relation between a cone and its cross section.

Let $\Sigma^{n-1}$ be an oriented embedded minimal hypersurface in $(M,g_M)$. And $C(\Sigma)$ is the cone whose cross section is $\Sigma$. Then we have: (see also \cite{Simons1968})
\begin{pp}\label{lem1}
    $C(\Sigma)$ is minimal in $C(M)$ iff $\Sigma$ is minimal in $M$.
\end{pp}
\begin{proof}
    Using the rules for connections, we can check the following holds:
    $$\overline{\nabla}_{\overline{X}}\overline{Y}=\nabla_XY-t\langle X,Y\rangle_{M}\frac{\partial}{\partial t},$$
    for all vector fields $X$ and $Y$ on $M$. Therefore, if $\{E_i\}$ is a basis of $\Sigma$, then the normal components are the same:
    $(\overline{\nabla}_{\overline{E_i}}\overline{E_i})^N=(\nabla_{E_i}E_i)^N$. And the conclusion follows by the definition of the mean curvature.
\end{proof}

In the end, we prove the claim in Proposition \ref{prop1}.
\begin{claim}
    Let $2\le n\le 6$, $0<\lambda<1$, and let $C(\Sigma^{n-1})$ be a minimal cone in $C(S^n(\lambda))$. If $C(\Sigma)$ is not totally geodesic, $i.e.$, $C(\Sigma)\neq C(S^{n-1}(\lambda))$, then $C(\Sigma)$ is unstable in $C(S^n(\lambda))$.
\end{claim}
\begin{proof}
    We can view $\Sigma^{n-1}$ as a closed minimal hypersurface in $S^n(\lambda)$ by the proposition above.
    
    Suppose on the contrary that $C(\Sigma)$ is stable. Then combining with (\ref{Ric}), the stability inequality yields:
    $$\int_0^{\infty}dt\int_{\Sigma_t}\left(\frac{\text{Ric}_{S^{n}(\lambda)}(\nu,\nu)-(n-1)}{t^2}+\lvert\overline{A}\rvert^2\right)\eta^2\le\int_{C(\Sigma)}\lvert\nabla^{C(\Sigma)}\eta\rvert^2,\quad \forall \eta \in C_0^{\infty}(C(\Sigma)). $$
    Here $\nu$ is the unit normal vector, $\overline{A}$ is the second fundamental form of $C(\Sigma)$ in $C(S^n(\lambda))$. 
    
    Besides, direct computation shows that:
    $$\text{Ric}_{S^{n}(\lambda)}(\nu,\nu)=\frac{n-1}{\lambda^2},\quad \lvert\overline{A}\rvert^2=\frac{\lvert A\rvert^2}{t^2},$$
    where $A$ is the second fundamental form of $\Sigma$ in $S^n(\lambda)$.
    
    Now we construct a new cone $C(\Sigma')$ embedded in the Euclidean space $C(S^n(1))=\mathbb{R}^{n+1}$, where $\Sigma'=\lambda^{-1}\Sigma$ is the scaling of $\Sigma$. Then obviously, $\Sigma'$ is minimal in $S^n(1)$, and the second fundamental form of $\Sigma'$ in $S^n(1)$ satisfies:
    $$\lvert A\rvert^2=\frac{\lvert A'\rvert^2}{\lambda^2}.$$
    Putting all these formulas together, we obtain:
    $$\int_0^{\infty}dt\int_{\Sigma_t}\frac{(n-1)(1-\lambda^2)+\lvert A'\rvert^2}{\lambda^2t^2}\eta^2\le\int_0^{\infty}dt\int_{\Sigma_t}\lvert\nabla^{C(\Sigma)}\eta\rvert^2,\quad \forall \eta \in C_0^{\infty}(C(\Sigma)). $$
    By the relationship between $\Sigma'$ and $\Sigma$,
    we have:
    $$\int_0^{\infty}dt\int_{\Sigma'_t}\frac{(n-1)(1-\lambda^2)+\lvert A'\rvert^2}{\lambda^2t^2}\lambda^{n-1}\phi^2\le\int_0^{\infty}dt\int_{\Sigma'_t}\frac{\lvert\nabla^{C(\Sigma')}\phi\rvert^2}{\lambda^2}\lambda^{n-1},$$
    where $\phi$ is the corresponding function on $C(\Sigma')$. And we use 
    $$\lvert\nabla^{C(\Sigma)}\eta\rvert^2\le\frac{\lvert\nabla^{C(\Sigma')}\phi\rvert^2}{\lambda^2},$$
    which can be checked directly by choosing an orthonormal basis on $\Sigma$.

    As a consequence, we have
    $$\int_0^{\infty}dt\int_{\Sigma'_t}\frac{\lvert A'\rvert^2}{t^2}\phi^2\le\int_0^{\infty}dt\int_{\Sigma'_t}\lvert\nabla^{C(\Sigma')}\phi\rvert^2,\quad \forall \in C_0^{\infty}(C(\Sigma)).$$
    Note that this is exactly the stablility inequality for $C(\Sigma')$ in $\mathbb{R}^{n+1}$. Therefore, by classical theory for minimal cones in Euclidean spaces (\cite{Giusti1984},\cite{Simon1983},\cite{Yau1975}), we know that $\Sigma'$ must be the totally geodesic $S^{n-1}(1)$ in $S^n(1)$. Thus, $C(\Sigma)$ must be the totally geodesic $C(S^{n-1}(\lambda))$, a contradiction.
\end{proof}

\section*{Acknowledgement}
The author is grateful to Prof. Gang Tian for his suggestions and encouragement. The author also thanks Prof. Sergey Kryzhevich, Shengxuan Zhou and Yaoting Gui for discussions on related problems. The author is supported by National Key R\&D Program of China 2020YFA0712800.

\bibliographystyle{plain}
\bibliography{Reference}

\end{document}